\documentclass[10pt,a4paper]{article}
\usepackage[utf8]{inputenc}
\usepackage{algorithm}
\usepackage{algpseudocode}
\usepackage{amsmath}
\usepackage{amsfonts}
\usepackage{amssymb}
\usepackage{amsthm}
\usepackage{bm}
\usepackage[giveninits=true,style=numeric]{biblatex}
\usepackage[font=small,labelfont=bf,indention=.5cm]{caption}
\usepackage[dvipsnames]{xcolor}
\usepackage{graphicx}
\usepackage[margin=2.5cm]{geometry} 
\usepackage[hidelinks]{hyperref} 
\usepackage[capitalize]{cleveref}
\usepackage{subcaption}
\usepackage{tikz} 
\usepackage{tikz-cd}
\usepackage{tikz-3dplot}
\usepackage{pgfplots} 
\usetikzlibrary{positioning,chains}

\usepackage[affil-it]{authblk}

\newtheorem{example}{Example}
\newtheorem{assumption}{Assumption}
\newtheorem{proposition}{Proposition}
\newtheorem{remark}{Remark}
\newtheorem{corollary}{Corollary}
\newtheorem{definition}{Definition}

\crefname{assumption}{Assumption}{Assumptions}
\crefname{example}{Example}{Example}
\crefname{proposition}{Proposition}{Propositions}
\crefname{definition}{Definition}{Definitions}

\newcommand*{\Break}{\textbf{break}}

\DeclareMathOperator{\Ker}{Ker}
\newcommand{\floor}[1]{\left\lfloor #1 \right\rfloor}

\newcommand{\lp}{\left(}
\newcommand{\rp}{\right)}

\DeclareMathOperator{\crit}{Crit}

\newcommand{\rd}{\mathbb{R}}

\newcommand{\relu}{\sigma_r}

\title{A singular Riemannian Geometry Approach to Deep Neural Networks III. Piecewise Differentiable Layers and Random Walks on $n$-dimensional Classes.}

\author[1]{Alessandro Benfenati}
\author[2]{Alessio Marta}

\affil[1]{\small{Environmental Science and Policy Department, Università di Milano},	{Via Celoria 2}, {Milano},
	{20133}, 
	{Italy}}

\affil[1]{{Gruppo Nazionale Calcolo Scientifico},
	{INDAM},{Italy}}   

\affil[2]{{Dipartimento di Scienze del Sistema Nervoso e del Comportamento, Universit{\`a} di Pavia},{Viale Golgi 19}, 
	{Pavia},
	{27100}, 
	{Italy}}

\date{}
          
\begin{document}
	
		\maketitle
	
	\begin{abstract}
		Neural networks are playing a crucial role in everyday life, with the most modern generative models able to achieve impressive results. Nonetheless, their functioning is still not very clear, and several strategies have been adopted to study how and why these model reach their outputs. A common approach is to consider the data in an Euclidean settings: recent years has witnessed instead a shift from this paradigm, moving thus to more general framework, namely Riemannian Geometry. Two recent works introduced a geometric framework to study neural networks making use of singular Riemannian metrics. In this paper we extend these results to convolutional, residual and recursive neural networks, studying also the case of non-differentiable activation functions, such as ReLU. We illustrate our findings with some numerical experiments on classification of images and thermodynamic problems.
	\end{abstract}

	\section{Introduction}\label{sec:intro}
	
	Deep Neural Networks play a crucial role in several tasks, such as computer vision for autonomous driving \cite{9440863,glei23}, image blind deconvolution \cite{9506502,benfenati2023neural}, image segmentation \cite{BenfenatiSSVM}, image restoration in medical frameworks \cite{9732379,sapienza2022deep,EVANGELISTA2023102156}, image classification \cite{a15100386}. They find applications also in processing of time series \cite{Tsa17} and natural language \cite{LAURIOLA2022443,Col08}, recommender systems \cite{Oord13} and brain-computer interfaces \cite{Avi20}. Moreover, the very recent period has witnessed an incredible impact of this architectures in everyday lives with new Large Language Models (LLMs) for text generation such as ChatGPT \cite{radford18} or BERT \cite{Tavan21}, and powerful networks are ow capable of creating realistic images \cite{geminiteam2023gemini} or videos \cite{liu2024sora}. 
	
	The literature presents several approach to study the behavior of such powerful architectures, by exploring the geometry of network's data, both in input and output and also between the hidden layers. The common approach is to consider that such data belong to an Euclidean space, with the natural notion of Euclidean distance \cite{kidger2020universal}. For example, \cite{NEURIPS2023_90043ebd} explores the space of images for studying the phenomenon of under-sensitivity of vision models, such as Convolutional Neural Networks (CNNs) and Transformers. The authors propose a Level Set Traversal algorithm, which construct the linear path between two different images that are classified with the same label by a vision model. This algorithm nonetheless assumes the Euclidean setting. The same geometry is used in \cite{salman2024intriguing}, where gradient-based procedures are employed for studying equivalence structures in the embedding space of Vision Transformers. 
	
	Assuming an Euclidean settings, one may lose some information: indeed, the data may live in a lower dimensional manifold, where the Euclidean distance may provide low information. A classical, intuitive, example is the one of an horseshoe in $\rd^3$: the two extrema of the horseshoe could be close, hence two points on different extrema could have a small Euclidean distance in the 3-dimensional space, but if such points must be connected via a path constrained to lie on the horseshoe manifold, the actual distance could be large. More practical examples encompass data lying on graphs, as 2D meshes in computer graphics \cite{boscaini2016learning} or weighted graphs, which occur in social network analysis \cite{CHUNAEV2020100286}. Recently, generative diffusion models experience this novel approach, \emph{i.e.}, employing the language of Riemannian manifolds \cite{de2022riemannian,fishman2023diffusion}.
	
	This approach goes under the term of Geometric Deep Learning \cite{9003285,7974879,henaff2015deep,8100059}: the idea to consider a network as a sequence of maps between manifold arose in \cite{hauser2017principles}. The work \cite{shen2018differential} inspired the first paper of this series \cite{BeMa21a} where via the pullback of the metric $g$ on the output manifold the authors were able to induce on all the manifolds $M_i, i=1,\dots,n$ of the network a structure of pseudometric space:
	\begin{equation*}
	\begin{tikzcd}
	M_{0} \arrow[r, "\Lambda_1"] & M_{1} \arrow[r, "\Lambda_2"] & M_{2} \arrow[r,"\Lambda_4"]  & \cdots \arrow[r,"\Lambda_{n-1}"] & M_{n-1} \arrow[r, "\Lambda_n"] & M_n
	\end{tikzcd}
	\end{equation*}
	$M_0$ takes the name of input manifold, while the last one, $M_n$, is the output manifold. The intermediate $M_2, M_3, \dots, M_n$ are representation manifolds. The maps $\Lambda_i$ are the functions mapping one manifold to another. This allows to obtain a full-fledged metric space, adopting a metric identification; beside the theoretical implications of this result (\emph{e.g.}, among others, the map $\Lambda_1\circ \dots \circ\Lambda_i$ preserve the length of the curves on the $i$-th manifold), practical results can be extracted from this framework. Indeed, building on the results of \cite{BeMa21a}, in \cite{BeMa21b} the authors developed 2 algorithms, namely SiMEC and SiMEXP: the former builds the classes of equivalence in the input manifold, that is the points that are mapped in the same output by the network, whilst the former allows to explore the input manifold by "jumping" from one class to another. Nonetheless, these works have some limits: the maps between manifolds are smooth, such as softmax, softplus, the models considered for the networks were only feedforward ones and the output of the networks were mono-dimensional.
	
	We extend hence here the results of \cite{BeMa21a}. In this work, we consider piecewise differential maps, such as Rectified Linear Unit (ReLU) and Leaky ReLU, which are commonly used as activation functions in Deep Learning techniques. Moreover, we generalize all the approaches to $n$-dimensional problems, such as image classification problems. The third novelty of this work consists in generalizing the previous theoretical results to convolutional operators and to complex network structures, such as residual and recurrent networks, used for \cite{Torr20} and \cite{Fan2015}.
	
	This work is organized as follows. \cref{sec:geometric_framework} firstly collects the basic notion of Riemannian geometry for making this paper self-contained (the interested reader may find more details in \cite{BeMa21a}). \cref{sec:random_walks_on_equivalence_classes} develop the strategy for the exploration of the equivalence classes via random walks, with some insights and comments about the probability to visit more than once the same inputs and the expected time for reaching a new class component. 
	\cref{sec:extension_c1_layers} generalizes the results of the previous paper to more complex architectures, such as convolutional layers, residual blocks and recurrent networks. \cref{sec:extension_non_differentiable} extends the framework to non differentiable functions. \cref{sec:numerical_experiments} presents numerical tests for the validation of the developed theory, and finally \cref{sec:concl} draws the final conclusions.

	\paragraph{Notation} $\mathbb{R}^n$ is real vector space, whose elements have $n$ elements. $\Ker(f)$ denotes the kernel of the linear application $f$. A map $f$ between two sets $A$ and $B$ is a function from $A$ to $B$. If $f: \mathbb{R}^m \rightarrow \mathbb{R}^m$ is a vector--valued function, we denote the $k$--th component of $f$ with $f_k$. Given a matrix $A$, we denote the element in row $i$ and column $j$ with $A_{ij}$; If $v$ is a vector in $\mathbb{R}^n$, we denote the $i$--th component with $v_i$. Given a metric $g$, the notation $g_{hk}$ denotes the element of the associated matrix at row $h$ and column $k$. The set $
	\mbox{ae--}\mathcal{C}^1(\Omega)$ is the set of almost-everywhere $\mathcal{C}^1$ functions:  $\mbox{ae--}\mathcal{C}^1(\Omega)= \{f \in \mathcal{C}^1(\Omega/Z), \,f\in\mathcal{C}^0(Z),\,  f\notin \mathcal{C}^1(Z), \,\mu(Z)=0\}
	$, with $\mu$ a suitable measure on $\mathbb{R}^n$.
	The Rectified Linear Unit function (ReLU) will be denoted via $\relu(x)$:
	$$
	\relu(x) = \begin{cases}
	x & x>0\\
	0 & \rm{otherwise}
	\end{cases}
	$$
	
	\section{A singular Riemannian geometry approach to neural networks}\label{sec:geometric_framework}
	
	In this section we start by a short resume on Riemannian geometry, and then we adapt this framework to neural network, recalling and generalizing some notions from \cite{BeMa21a,BeMa21b}. As stated in \Cref{sec:c_k} the results of \cite{BeMa21a} keep to hold true considering functions which are just differentiable. Furthermore, in \Cref{sec:extension_non_differentiable,sec:extension_c1_layers} we shall extend this framework to convolutional layers, residual blocks, recurrent network and non-differentiable layers. 
	
	\subsection{Singular Riemannian metrics}

	we point out that the manifolds we are going to consider will be either $\mathbb{R}^n$ or some open subsets of $\mathbb{R}^n$, $n \in \mathbb{N}$. For the case of a generic $n$-dimensional smooth manifold, the interested reader can find more details in \cite{BeMa21a} where we treat singular Riemannian metrics in full generality. Intuitively, a singular metric $g$ can be seen as a degenerate scalar product changing from point to point.
	\begin{definition}[Singular Riemannian metric]\label{def:simplified_singular_metric}
		Let $M=\mathbb{R}^n$ or an open subset of $\mathbb{R}^n$.
		A singular Riemannian metric $g$ over $M$ is a map $g:M \rightarrow Bil(\mathbb{R}^n \times \mathbb{R}^n)$ that associates to each point $p$ a positive semidefinite symmetric bilinear form $g_p:\mathbb{R}^n \times \mathbb{R}^n \rightarrow \mathbb{R}$ in a smooth way.
	\end{definition}
	We shall use pseudometric or singular metric for referring to the same mathematical concept of \cref{def:simplified_singular_metric}.
	\begin{remark}
		In \cref{def:simplified_singular_metric}, we identified $\mathbb{R}^n$ as an affine space with its space of displacement vectors -- or, in more general terms, the smooth manifold $\mathbb{R}^n$ and its tangent space at each point \cite{Tu11,DoCarmo16}. However, one should keep in mind that a metric $g$ associates to every point $p$ in the affine space a bilinear form over the space of displacement vectors.
	\end{remark}
	Note that the singular metric $g_p(x,y)$ at a point $p \in \mathbb{R}^n$ may be null even if both $x\neq 0$ and $y\neq0$. Given a vector $v \in \mathbb{R}^n$, we define the semi--norm of a $\| v \|_p = \sqrt{g_p(v,v)}$. Given a curve $\gamma:[a,b]\rightarrow \mathbb{R}^n$, we can define its pseudolenght.
	\begin{definition}[Pseudolenght of a curve]
		\label{def:pseudolength}
		Let $\gamma:[a,b]\rightarrow \mathbb{R}^n$ a curve defined on the interval $[a,b]\subset\mathbb{R}$ and $\|v\|_p=\sqrt{g_p(v,v)}$ the pseudo--norm induced by the pseudo--metric $g_p$ at point $p$. Then the pseudolength of $\gamma$ is defined as
		\begin{equation}
		Pl(\gamma) = \int_a^b \| \dot{\gamma}(s) \|_{\gamma(s)} ds = \int_a^b \sqrt{g_{\gamma(s)}(\dot{\gamma}(s),\dot{\gamma}(s))} ds
		\end{equation}
	\end{definition}
	Another useful notion, closely related to the pseudolenght of a curve is that of energy of a curve.
	\begin{definition}[Energy of a curve]
		\label{def:energy}
		Let $\gamma:[a,b]\rightarrow \mathbb{R}^n$ a curve defined on the interval $[a,b]\subset\mathbb{R}$ and $\|v\|_p=\sqrt{g_p(v,v)}$ the pseudo--norm induced by the pseudo--metric $g_p$ at point $p$. Then the energy of $\gamma$ is defined as
		\begin{equation}
		E(\gamma) = \int_a^b g_{\gamma(s)}(\dot{\gamma}(s),\dot{\gamma}(s)) ds
		\end{equation}
	\end{definition}
	A notable consequence of the degeneracy of the metric, is that there may exist smooth non-constant curves which have null length, called \emph{null curves}. This happens when $\dot{\gamma}(s) \in \Ker(g_{\gamma(s)})$ for every $s \in [a,b]$. The notion of pseudolenght allows us to equip $\mathbb{R}^n$ with the structure of pseudometric space with the pseudodistance defined in \cref{def:pseudodistance}.
	\begin{definition}[Pseudodistance]
		\label{def:pseudodistance}
		Let $x,y \in\mathbb{R}^n$. The pseudodistance between $x$ and $y$ is then 
		\begin{equation}\label{eq:pseudodistance}
		Pd(x,y) = \inf \{ Pl(\gamma) \ | \ \gamma:[0,1]\rightarrow M \mbox{ is a piecewise } \mathcal{C}^1 \mbox{ curve with } \gamma(0)=x \mbox{ and } \gamma(1)=y \}  
		\end{equation}
		where $Pl(\gamma)$ denotes the pseudolength of the curve $\gamma$ as in \Cref{def:pseudolength}.
	\end{definition}
	As we observed before, in $\mathbb{R}^n$ endowed with a singular Riemannian metric, there exist non-trivial curves of length zero. A notable consequence is that there are points whose distance is null, which are metrically indistinguishable. Identifying these points making use of the equivalence relation $x \sim y \Leftrightarrow Pd(x,y)=0$ for $x,y \in \mathbb{R}^n$, we obtain a metric space $(\mathbb{R}^n / \sim,Pd)$. Given a point $x \in \mathbb{R}^n$ its class of equivalence is of the form $[x]=\{y \in \mathbb{R}^n \ | \ Pd(x,y) = 0\}$, see \cite{BeMa21a} for further details). The last key ingredient of our work is the notion of pullback of a metric through a $\mathcal{C}^k$ map. Let $F:\mathbb{R}^n \rightarrow \mathbb{R}^m$ a smooth function. Suppose to endow $\mathbb{R}^m$ with the canonical Euclidean metric $g $ whose associated matrix is $\bm{I}_m$, the identity matrix of dimension $m$. Then we can 	equip $\mathbb{R}^n$ with the pullback metric $F^*g$ as follows. Chosen two coordinate systems $(x_1,\cdots,x_m)$ and $(y_1,\cdots,y_n)$ of $\mathbb{R}^m$ and $\mathbb{R}^n$ respectively, the matrix associated to the pullback of $g$ through $F$ reads:
	\begin{equation}\label{eq:pullback}
	(F^*g)_{ij} = \sum_{h,k = 1}^{m} \left( \frac{\partial F^h}{\partial x^i} \right) g_{hk} \left( \frac{\partial F^k}{\partial x^j} \right)
	\end{equation}
	We use the pullback to transport the known metric information of the codomain of a function $F$ to the domain. In particular, in many cases, the pullback metric allows to identify all the points of the domain which are mapped to the same element of the codomain.
	At last we recall the notion of submersion.
	\begin{definition}[Submersion]\label{def:submersion}
		Let $f:M \rightarrow N$ be a $\mathcal{C}^1$ map between manifolds. Then $f$ is a submersion if, in any chart, the Jacobian $J_F$ has rank $\dim(N)$.
	\end{definition}
	
	\subsection{Quotients induced by the degenerate metrics}\label{sec:c_k}
	Let $f:D \subseteq \mathbb{R}^n\rightarrow \mathbb{R}$ be a smooth function. Given $x,y \in D$, the notion of connected level curve introduce the equivalence relation $x \sim y$ if and only if there is a piecewise $\mathcal{C}^1$ curve $\gamma:[0,1]\rightarrow \mathbb{R}^n$ with $\gamma(0)=x$ and $\gamma(1)=y$ such that $f(\gamma(s))=f(x)$ for every $s \in [0,1]$. The metric $g=(1)\in\rd^{1 \times 1}$, \emph{i.e.}, the identity, is the trivial one over $\mathbb{R}$, the pullback $f^*g$ is a singular Riemannian metric over $D$. The presence of a singular metric canonically introduce another equivalence relation, $\sim_g$, defined as follows: $x \sim_g y$ if $x$ and $y$ are connected by a null curve. By \cite[Proposition 4]{BeMa21a} the two spaces $D/\sim$ and $D/\sim_g$ coincide, therefore we can characterize the connected components of the level sets as the sets of all the connected points in $D$ whose pseudodistance is null. However, we still do not know if this space is still a smooth manifold. Assuming the additional hypothesis that $f$ is also a submersion, namely $\nabla f(x) \neq 0 \ \forall x \in D$, then \cite[Proposition 6 and Proposition 7]{BeMa21a} hold true and we can conclude that:
	\begin{proposition}\label{prop:smooth}
		Let $f:D \subseteq \mathbb{R}^n\rightarrow \mathbb{R}$ be a smooth submersion. The connected components of the level sets of $f$ are path connected submanifolds of $D$ of dimension $n-1$, whose tangent vectors are in $Ker(f^*g)$.
	\end{proposition}
	As a matter of fact both \cite[Proposition 6]{BeMa21a} and \cite[Proposition 7]{BeMa21a} keep to hold true even if we assume that the function $f$ is only $\mathcal{C}^1$. Indeed, in the proof of Godement's criterion, which is employed to prove the two aforementioned propositions, assuming to work with a $\mathcal{C}^1$ function allows to prove that there is a differentiable structure of class $\mathcal{C}^1$ instead of one of class $\mathcal{C}^\infty$. Since for our purposes we only need to know the first derivatives, the existence of a $\mathcal{C}^1$ differentiable structure over a topological manifold is all we need. Note that this observation allows to relax the smoothness hypothesis also in \cite[Proposition 4]{BeMa21a}. We also note that in the case in which $f$ is not a submersion, \Cref{prop:smooth} holds true for $f$ restricted to $D \setminus \crit(f)$, with $\crit(f) := \{x \in D \ | \ \triangledown f = 0 \}$. Therefore we can generalize \Cref{prop:smooth} as follows.
	\begin{proposition}\label{prop:c_1}
		Let $f:D \subseteq \mathbb{R}^n\rightarrow \mathbb{R}$ be a $\mathcal{C}^1$ function. Let $h := f|_{E}, \, E = D \setminus \crit(f)$. Then the connected components of the level sets of $h$ are path connected submanifolds of $E$ of dimension $n-1$, whose tangent vectors are in $Ker(h^*g)$.
	\end{proposition}

	\subsection{The Geometric Framework for Neural Networks}

	\begin{definition}[Neural Network]\label{def:neural_netowork}
		A neural network is a sequence of $\mathcal{C}^1$ maps $\Lambda_i$ between manifolds of the form:
		\begin{equation}\label{def:sequence_of_maps}
		\begin{tikzcd}
		M_{0} \arrow[r, "\Lambda_1"] & M_{1} \arrow[r, "\Lambda_2"] & M_{2} \arrow[r,"\Lambda_4"]  & \cdots \arrow[r,"\Lambda_{n-1}"] & M_{n-1} \arrow[r, "\Lambda_n"] & M_n
		\end{tikzcd}
		\end{equation}
		We call $M_0$ the \textit{input manifold} and $M_n$ the \textit{output manifold}. All the other manifolds of the sequence are called \textit{representation manifolds}. The maps $\Lambda_i$ are the layers of the neural network.
	\end{definition}
	The assumptions on this sequence of maps are the following.
	\begin{assumption}\label{Assumption_1}
		The manifolds $M_i$ are open and path-connected sets of dimension $\dim M_i=d_i$.
	\end{assumption}
	\begin{assumption}\label{Assumption_2}
		The sequence of maps \eqref{def:sequence_of_maps} satisfies the following properties:
		\begin{itemize}
			\item[1)] The maps $\Lambda_i$ are $\mathcal{C}^1$ submersions.
			\item[2)] $\Lambda_i(M_{i-1})=M_i$ for every $i = 1, \cdots , n$.
		\end{itemize}
	\end{assumption}
	\begin{remark}\label{rem:dimensions}
		In our framework the dimension of the manifold $M_i$ does not correspond to the number of nodes of a layer. See \cite{BeMa21a} for a thorough discussion. Note that this assumption entails that $\dim(M_i) \geq \dim(M_{i+1})$.
	\end{remark}
	\begin{assumption}\label{Assumption_3}
		The manifold $M_n$ is equipped with the structure of Riemannian manifold, with metric $g^{(n)}$.
	\end{assumption}
	The pullbacks of $g_n$ trough the maps $\Lambda_n$, $\Lambda_n \circ \Lambda_{n-1}$, ..., $\Lambda_n \circ \Lambda_{n-1} \circ \cdots \circ \Lambda_{1}$ yield a sequence of (in general degenerate) Riemannian metrics $g^{(n-1)}, g^{(n-2)},\cdots, g^{(0)}$ on $M_{n-1},M_{n-2},\cdots,M_0$.
	\begin{remark}
		In the rest of the paper, we shall denote with $\mathcal{N}_i$ the map $\Lambda_n \circ \Lambda_{n-1} \circ \cdots \circ \Lambda_{i}$. In the case $i=1$, we shall simply write $\mathcal{N}$ instead of $\mathcal{N}_1$, since we are considering the whole neural network.
	\end{remark}
	By \Cref{prop:c_1} an immediate consequence of our definitions is that two points in an equivalence class of any manifold $M_i$ induced by the pullback metric $g^{(i)}$ are mapped by the subsequent layers, namely by $\Lambda_i \circ \cdots, \Lambda_n$ to the same output.
	In \cite{BeMa21a} we considered smooth feedforward layers, that we modeled as a particular kind of maps between manifolds defined as follows.
	\begin{definition}[Smooth layer]\label{def_layer}
		Let $M_{i-1}$ and $M_i$ be two smooth manifolds satisfying the assumptions above. A map $\Lambda_i : M_{i-1} \rightarrow M_i$ is called a smooth layer if it is the restriction to $M_{i-1}$ of a function $\overline{\Lambda}^{(i)}(x)  : \mathbb{R}^{d_{i-1}} \rightarrow \mathbb{R}^{d_i}$ of the form
		\begin{equation}
		\overline{\Lambda}^{(i)}_\alpha (x)= F^{(i)}_\alpha\lp\sum_\beta A_{\alpha\beta}^{(i)} x_\beta+b^{(i)}_\alpha\rp 
		\end{equation}
		for $i=1,\cdots,n$, $x \in \mathbb{R}^{d_{i}}$, $b^{(i)} \in \mathbb{R}^{d_i}$ and $A^{(i)} \in \mathbb{R}^{d_{i} \times d_{i-1}}$, with $F^{(i)} : \mathbb{R}^{d_i} \rightarrow \mathbb{R}^{d_i} $ a diffeomorphism.
	\end{definition}
	For these kind of layers we also assumed the full rank hypothesis (see \cite[Remark 8]{BeMa21a} for a throughout discussion about this hypothesis).

	In \cref{sec:extension_c1_layers} we will consider other feedforward structures like convolutional layers and residual blocks. Our framework can also be adapted to recurrent neural networks: In the case a neural network can retain memory of its previous states we can unfold it in time, obtaining a sequence of states of the neural network -- along with the input data from the previous states. This topic is addressed in \cref{ssec:recurrent_layers_and_networks}. In \cref{sec:extension_non_differentiable} we will relax the differentiability assumption on the layer to treat activation functions like ReLU.
.

	\section{Random walks on (and between) n-dimensional equivalence classes}\label{sec:random_walks_on_equivalence_classes}
	
	In \cite{BeMa21b} we discussed how to build the one-dimensional equivalence classes of some kind of neural networks whose output manifold is of dimension one, presenting in particular the SiMEC algorithm. This section is devoted to generalize these results to the case of n-dimensional equivalence classes. In the case of $1D$ equivalence classes, the space of the the null vectors at a point is a line; If instead we deal with equivalence classes of generic dimension $n$, the null vectors at each point are generating a vector space of dimension $n$. In both cases starting from the point $p$ and moving along null vectors \footnote{With moving along a null vector we mean moving along the integral curve whose starting point is $p$ and whose initial velocity is the chosen null vector at $p$.} we remain on the same equivalence class (see \cref{fig:ComparisonEqClasses}).
	\begin{figure}[h!]
		\begin{center}
			\begin{subfigure}[t]{0.45\textwidth}
				\begin{center}
					\includegraphics[width=.95\linewidth]{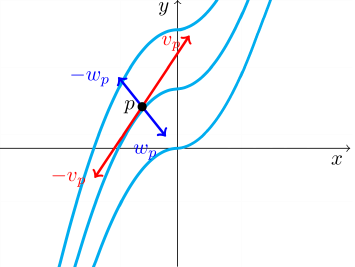}
					\caption{$2D$ manifold with its $1D$ equivalence classes.}
					\label{fig:1d_equivcl}
				\end{center}
			\end{subfigure}\hfill\begin{subfigure}[t]{0.45\textwidth}
				\begin{center}
					\includegraphics[width=.95\linewidth]{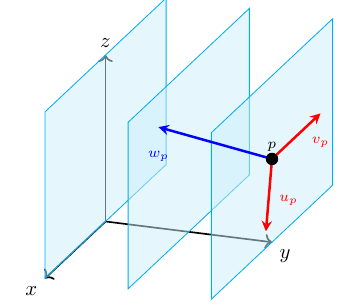}
					\caption{$3D$ manifold with its $2D$ equivalence classes.}
					\label{fig:2d_equivcl}
				\end{center}
			\end{subfigure}
		\end{center}
		\caption{Equivalence classes and null vectors of a manifold $M$. Given a point $p$, by proceeding in the direction of the null vectors $\pm v_p$ we stay on the class of equivalence $[p]$. By proceeding in the direction of a non-null vector $\pm w_p$, we arrive to another class of equivalence. Left panel: A $2D$ manifold is foliated by one dimensional equivalence classes (the cubic curves in cyan). At the point $p$ there is a line of null vectors (in red), proceeding along which we stay on the same class of equivalence. Moving in the direction of non-null vectors (in blue) we change class of equivalence. The tangent plane at a point of a $2D$ manifold is a vector space of dimension $2$, therefore since the space of the null vectors is one dimensional, also the space of the non-null vectors is one dimensional. Right panel: A $3D$ manifold is foliated by $2D$ equivalence classes (the cyan planes). This time the space of the null vectors is a two-dimensional space -- moving along linear combinations of $u_p$ and $v_p$ we remain on the same equivalence class of $p$, the right cyan plane. Moving along the direction of $w_p$, which is non-null, we change equivalence class. The space of the non null vectors is one dimensional. In both cases the manifold $M$ is foliated by its classes of equivalence.}
		\label{fig:ComparisonEqClasses}
	\end{figure}
	Theoretically, we could explore a $nD$ equivalence class generating a mesh of point following at each point $n$ linearly independent null vectors generating the $n$-dimensional vector space of null-vectors: see \cref{fig:meshExpl} for a visual inspection. Starting from a point $p$, in this way we find $2^n$ points on the same equivalence class, for each of them we find other $2^n$ points following their null vectors and so on. After $k$ iterations we generated $2^{n\,k}$ points on the same equivalence class. It is clear that if the dimension $n$ is not low enough, let us say $n=2,3$, and if the the number of iterations is not small computational issues about running out of memory will surely arise. Take for example the MNIST handwritten digit database \cite{MNIST} with $28 \times 28$ monochromatic images ($n=28 \cdot 28 = 784$). Consider the usual classification task with $10$ different classes, one per digit. The dimension of the equivalence classes is then $784-10=774$. After $10$ iterations of the procedure we just described we find $2^{7740}$ points, namely about $10^{2329}$ points - more than the number of atoms in the entire known universe! A convenient approach to overcome this hurdle, another example of the curse of dimensionality, is to generate just one point per iteration randomly choosing a null vector each time and therefore performing a random walk on the equivalence class. See \Cref{fig:randomexp} for such strategy.
	\begin{figure}[htbp]
		\begin{center}
			\begin{subfigure}[t]{0.45\textwidth}
				\begin{center}
					\includegraphics[width=.95\linewidth]{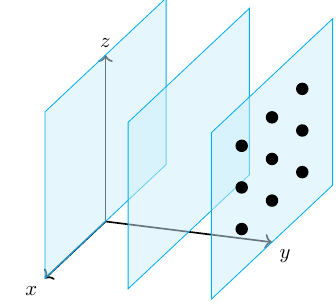}
					\caption{Exploration via a mesh grid.}
					\label{fig:meshExpl}
				\end{center}
			\end{subfigure}\hfill\begin{subfigure}[t]{0.45\textwidth}
				\begin{center}
					\includegraphics[width=.95\linewidth]{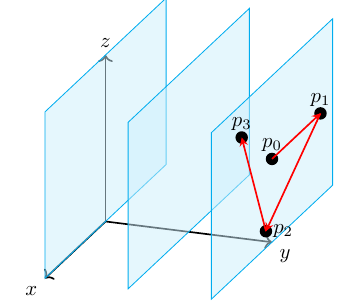}
					\caption{Random walk exploration.}
					\label{fig:randomexp}
				\end{center}
			\end{subfigure}
		\end{center}
		\caption{A comparison between the two approaches. Left panel: In the first case we build a $n$-dimensional grid of points approximating the equivalence class with a $n-$dimensional mesh. In a single iteration from each point we found at the previous step we build other $n$ new points -- the number of points grows exponentially with the number of iterations and the dimension of the space. Right panel: In the second approach, we generate points on a path described by randomly chosen null vectors -- the number of points to generate is independent on the dimension of the equivalence classes and increases linearly with the number of iterations.}
		\label{fig:GridVsRandomWalk}
	\end{figure}
	Employing a random walk-based exploration strategy raises some points. For example, this exploration could step on the same element more than once: but, depending on the dimension of the class, the probability of this event is small. For example, in the MNIST case, the probability to visit the same image is around 0.06\%, obtained via the theoretical estimation for multidimensional random walk in \cite{Montroll56}, which employs modified Bessel functions. Moreover, the time $\tau$ needed for visiting new sites by passing on already visited one is less than one, meaning that actually each random step visits a new sites (see \cite[table 1]{Regnier22} for the theoretical details).
	
	Given the data above, the steps of the SiMEC-nD algorithm to build an approximation of $[p] \subset M_0$ are depicted in \Cref{al:SIMECnD}.
	
	\begin{algorithm}[H]
		\begin{algorithmic}
			\Require {Choose $p_0 \in M_0$, $\delta>0$, maximum number of iterations $K$.}
			\Ensure  {A sequence of points $\{p_s\}_{s=1,\dots,K}$ approximatively in $[p_0]$; The energy $E$ of the approximating polygonal}
			\State{Initialise the energy: $E\gets 0$}
			\For {$k=1,\dots,K-1$}
			\State {Compute $g_{\mathcal{N}(p_k)}^n$}
			\State {Compute the pullback metric $g^0_{p_k}$ trough \Cref{eq:pullback}}
			\State {Diagonalize $g^0_{p_k}$ and find the eigenvectors}
			\State {$v_k \gets$ Choose a random linear combinations of the null eigenvectors}  
			\State {Compute the new point $p_{k+1} \leftarrow p_{k}+ \delta v_{k}$}
			\State {Add the contribute of the new segment to the energy $E$ of the polygonal}
			\EndFor
		\end{algorithmic}
		\caption{The SiMEC nD random walk algorithm }\label{al:SIMECnD}
	\end{algorithm}
	\begin{remark}
		The same algorithm can be applied to build the equivalence classes of a representation manifold $M_i$. The only difference is that we compute the pullback $g^i_{p_k} = \left( \Lambda_n \circ \Lambda_{n-1} \circ \cdot \circ \Lambda_{i+1} \right)^* g_n$ instead of the pullback through the entire network.
	\end{remark}
	\cref{al:SIMECnD} is not free from numerical and approximation errors: in order to reduce the impact of such errors several strategies have been employed, see \cite[Section 4.1]{BeMa21b}.
	\begin{remark}
		The energy of a curve introduced in \cref{def:energy} might be employed in \cref{al:SIMECnD} and the subsequent ones as a stopping criterion: once the energy is below a preset threshold, the algorithm stops.
	\end{remark}
	
	In \cite{BeMa21b} we introduced also the the SiMExp algorithm allowing one to pass from a given equivalence class to a near one at each step. \cref{al:SIMEXPnD} presents the generalization of the SiMExp algorithm, following the same approach, \emph{i.e.}, a random walk strategy, adopted for SiMEC. A visual inspection is given in \cref{fig:SiMExp}.

	\begin{algorithm}
		\begin{algorithmic}
			\Require {Choose $p_0 \in M_0$, $\delta>0$, maximum number of iterations $K$.}
			\Ensure  {A sequence of points $\{p_s\}_{s=1,\dots,K}$ such that two consecutive points are not in the same equivalence classes; The length $\ell$ of the approximating polygonal}
			\State{Initialise the length: $\ell\gets 0$}
			\For {$k=1,\dots,K-1$}
			\State {Compute $g_{\mathcal{N}(p_k)}^n$}
			\State {Compute the pullback metric $g^0_{p_k}$ trough \Cref{eq:pullback}}
			\State {Diagonalize $g^0_{p_k}$ and find the eigenvectors}
			\State {$v_k \gets$ Choose a random linear combinations of the non-null eigenvectors}
			\State {Compute the new point $p_{k+1} \leftarrow p_{k}+ \delta v_{k}$}
			\State {Add the contribute of the new segment to the length $\ell$ of the polygonal}
			\EndFor
		\end{algorithmic}
		\caption{SiMExp-nD random walk algorithm}\label{al:SIMEXPnD}
	\end{algorithm}
	
	\begin{figure}[h!]
		\begin{center}
			\begin{subfigure}[t]{0.45\textwidth}
				\begin{center}
					\includegraphics[width=.95\linewidth]{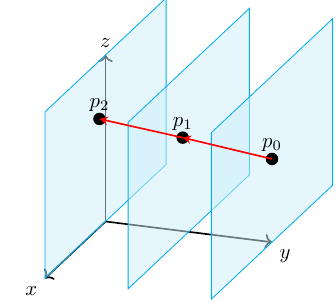}
					\caption{Moving between $2D$ classes of a $3D$ manifold.}
					\label{fig:smiexp1}
				\end{center}
			\end{subfigure}\hfill\begin{subfigure}[t]{0.45\textwidth}
				\begin{center}
					\includegraphics[width=.95\linewidth]{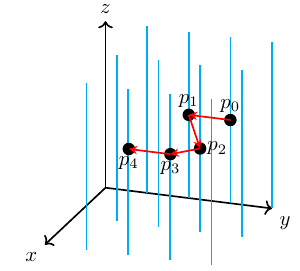}
					\caption{Moving between $1D$ classes of a $3D$ manifold.}
					\label{fig:smiexp2}
				\end{center}
			\end{subfigure}
		\end{center}
		
		\caption{The random walks generated by the SiMExp algorithm in a $3D$ manifold with equivalence classes of different dimensions. Left panel: The equivalence classes are planes and the algorithm yields a sequence of points lying on different planes. Right panel: The equivalence classes are lines and the algorithm yields a sequence of points lying on different lines. In both cases there is also the possibility during the random walk to visit again an equivalence class which has been already visited before.} 
		\label{fig:SiMExp}
	\end{figure}
	
	\begin{remark}
		As we will see in \Cref{sec:extension_non_differentiable}, with non differentiable activation functions, such as ReLU, the dimension of the equivalence classes is not the same, therefore the number of linearly independent null vectors can change from an equivalence class to another one.
	\end{remark}
	As already shown in \cite{BeMa21b}, one can combine \cref{al:SIMECnD} and \cref{al:SIMEXPnD} for the exploration of a manifold in its wholeness, moving among different equivalence classes and inside the same equivalence class.
	
	\section{$\mathcal{C}^1$ layers and networks of common use}\label{sec:extension_c1_layers}
	This section is devoted to generalize the results of \cite{BeMa21a} to several kind of differentiable layers and blocks of common use. Here we consider differentiable activation function, and we are going to discuss on layers coupled with non differentiable activation functions, such as ReLU and Leaky ReLU, in \Cref{sec:extension_non_differentiable}.
	
	\subsection{Convolutional layers}
	The first kind of layer we consider to extend our framework is the convolutional layer, mainly employed in neural networks processing visual data \cite{LeCun98}, ranging from images classification \cite{LeCun1995} to obstacle detection \cite{glei23} but that find applications also in processing of time series \cite{Tsa17} and natural language \cite{Col08}, recommender systems \cite{Oord13} and brain-computer interfaces \cite{Avi20}. For simplicity we treat in full detail the case of monochromatic images. The case of RGB images will be treated in essentially the same way. A similar approach has been used in \cite{peleska2023}, but the results of this work hold in more general settings and the application in \Cref{subsec:exp_mnist} is new. Consider a monochromatic image of dimension $n \times n$. This is represented by a matrix $P$ of pixels  which can assume values in a certain range, let us say from $0$ to $255$. Then we can see the image as a subspace of the set of the $n \times n$ matrices, which is naturally isomorphic to $\mathbb{R}^{n^2}$. In particular, we can flatten the 2D data $P_{i,j}$ into a vector $v \in \mathbb{R}^{n^2}$ as follows
	\begin{equation}
	v_j = P_{\floor{j/n} \ , \ j \ mod \ n}
	\end{equation}
	for $j = 0,\ldots,n-1$, where $\floor{ \ }$ is the floor operation.
	Conversely if $v \in \mathbb{R}^{n^2}$, we can go back to a 2D image by means of
	\begin{equation}
	P_{i,j} = v_{n \cdot i+j}
	\end{equation}
	with $i,j = 0, \ldots n-1$. Now we can transform the output of the convolution operation using an analogous isomorphism. 
	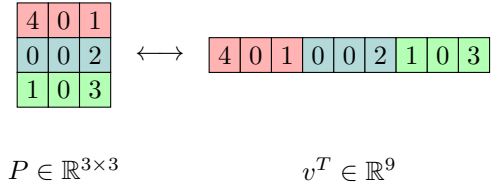
\begin{figure}[htbp]
		\begin{center}
			\begin{tikzpicture}[
			2d-arr/.style={matrix of nodes, row sep=-\pgflinewidth, column sep=-\pgflinewidth, nodes={draw}}
			]
			\matrix (P) [2d-arr] {
				|[fill=red!30]| 4 & |[fill=red!30]| 0 & |[fill=red!30]|1\\
				|[fill=teal!30]| 0 & |[fill=teal!30]| 0 & |[fill=teal!30]| 2\\
				|[fill=green!30]| 1 & |[fill=green!30]| 0 & |[fill=green!30]| 3\\
			};
			\node[below=of P-2-2] {$P \in \mathbb{R}^{3\times3}$};
			\node[right=0.2em of P] (str) {$\longleftrightarrow$};
			
			\matrix (v) [2d-arr, right=0.2em of str, nodes={draw, fill=yellow!30}] {
				|[fill=red!30]| 4 & |[fill=red!30]| 0 & |[fill=red!30]| 1 & |[fill=teal!30]| 0 & |[fill=teal!30]| 0 & |[fill=teal!30]| 2 &  |[fill=green!30]| 1 & |[fill=green!30]| 0 & |[fill=green!30]| 3 \\
			};
			\node[below=of v-1-5] {$v^T\in \mathbb{R}^9$};
			\end{tikzpicture}
		\end{center}
		\caption{A $3 \times 3$ matrix is flattened into a vector of length $9$. The spaces $\mathbb{R}^{3\times3}$ of $3 \times 3$ matrices and $\mathbb{R}^9$ are isomorphic. 
		}
	\end{figure}
	Let us begin revising the standard procedure for $2D$ images. Given an odd $k$, the convolution kernel $\mathcal{K}$ is a $k \times k$ matrix acting on the neighbourhood of a given pixel in the following way. Consider the pixel at the entry $a,b$ of the matrix $P$ representing the image. First, we build a new $k \times k$ matrix $\mathcal{S}$ from $P$ taking the $k \times k$ submatrix centered at $i,j$, namely $\mathcal{S}_{ij}=P_{ij}$ for $i = a-k,a-k+1,\ldots,a+k$ and $j = b-k,b-k+1,\ldots,b+k$. Then we compute the Frobenius product between $\mathcal{K}$ and $\mathcal{S}$. At this point we proceed choosing another pixel to build a new matrix $\mathcal{S}$, a choice depending of the size of the stride. We repeat the procedure until we cover the whole image. If no padding is employed and the dimension of the stride is $d_s$, then we build a sequence of matrices $\mathcal{S}$ centered at the pixel $ \floor{k/2} + \alpha d_s, \floor{k/2} + \beta d_s$ with $\alpha,\beta = 1,2,\ldots,\floor{n/d_s}$. In this case the output $\mathcal{O}$ is a $\floor{n/d_s} \times \floor{n/d_s}$ matrix. See \Cref{fig:convolution} for a visual representation of these operations. If we want to obtain an output matrix of the same dimension of the input one, we can pad the input matrix with zeros (or any other value of choice, depending on the chosen boundary condition) on the border of the input volume. In any case, after we apply some $\mathcal{C}^1$ nonlinear functions on the entries of $\mathcal{O}$, we can flatten the resulting matrix to obtain the output vector. 
	\begin{figure}
		\begin{center}
			\begin{tikzpicture}[
			2d-arr/.style={matrix of nodes, row sep=-\pgflinewidth, column sep=-\pgflinewidth, nodes={draw}}
			]
			\matrix (P) [2d-arr] {
				|[fill=red!30]| 1 & |[fill=red!30]| 0 & |[fill=red!30]| 4 & 0 & 1\\
				|[fill=red!30]| 3 & |[fill=red!30]| 1 & |[fill=red!30]| 0 & 0 & 2\\
				|[fill=red!30]| 7 & |[fill=red!30]| 1 & |[fill=red!30]| 1 & 0 & 3\\
				8 & 1 & 5 & 0 & 0\\
				5 & 2 & 0 & 0 & 0\\
			};
			
			\node[left=of P-1-3] {$\mathcal{S}$};
			\node[below=of P-4-3] {$P$};
			\node[right=0.2em of P] (str) {$*$};
			
			\matrix (K) [2d-arr, right=0.2em of str, nodes={draw, fill=yellow!30}] {
				1 & 0 & 0 \\
				0 & 1 & 0 \\
				0 & 0 & 1 \\
			};
			\node[below=of K-3-2] {$\mathcal{K}$};
			\node[right=0.2em of K] (eq) {$=$};
			
			\matrix (ret) [2d-arr, right=0.2em of eq] {
				|[fill=orange!80!black!30]| 3 & 7\\
				8 & 1\\
			};
			\node[below=4.8em of ret-1-2] {$\mathcal{O} = P * \mathcal{K}$};
			
			\draw[dashed, blue] (P-1-3.north east) -- (K-1-1.north west);
			\draw[dashed, blue] (P-3-3.south east) -- (K-3-1.south west);
			
			\draw[dashed, blue] (K-1-3.north east) -- (ret-1-1.north west);
			\draw[dashed, blue] (K-3-3.south east) -- (ret-1-1.south west);
			\end{tikzpicture}
		\end{center}
		\caption{Pictorial representation of the operations performed by a convolutional layer. $P$ is a $5 \times 5$ matrix representing a monochromatic picture, $\mathcal{K}$ is a $3 \times 3$  convolution kernel and the features map $\mathcal{O}$ is the result of the convolution operation $P * \mathcal{K}$, a $\floor{5/2} \times \floor{5/2} = 2 \times 2$ matrix. The dimension of the stride is $2$ and no padding is employed. To obtain the entry in orange of $\mathcal{O}$, we compute the Frobenius product of the red ($\mathcal{S}$) and yellow ($\mathcal{K}$) matrices.} 
		\label{fig:convolution}
	\end{figure}
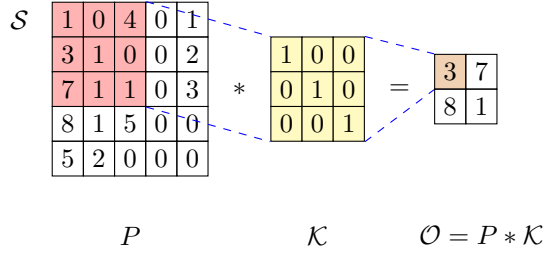
	The operations we employed in this procedure are compositions of linear applications (the flattening of the data) and $\mathcal{C}^1$ functions. Therefore we built a $\mathcal{C}^1$ map between manifolds and can apply the results of \Cref{sec:c_k}.
	\begin{remark}
		We use the construction above to justify the fact that we can apply our framework to convolutional layers. In the numerical experiments of \cref{sec:numerical_experiments} under the Pytorch environment we employ convolutional layers in the usual sense, with both the input and the convolution kernel considered as matrices, since automatic differentiation yields directly the desired result for the pullback on the input data manifold, automatically realizing the isomorphism between matrices and vectors.
	\end{remark}
	The case for RGB images is analogous. Flattening the matrices of each channel into three vectors of length $n^2$ and concatenating them in one vector of $\mathbb{R}^{3n^2}$ -- in other words, flattening the tensor containing the image -- we built an isomorphism between the space of RGB images and $\mathbb{R}^{3n^2}$. At this point the same arguments we employed for monochromatic images apply.

	\begin{remark}\label{remark:pooling}
		It is very common to use a convolutional layer together with a pooling layer to create a downsampled feature map. In our framework we can treat only differentiable activation functions -- the exceptions being the cases discussed in \Cref{sec:extension_non_differentiable}, which includes ReLU and leaky ReLU -- therefore we cannot use max pooling layers, which are not even continuous. However, we can consider average pooling layer or we may decide to employ smooth approximations of the max function, like softmax. For this reason in the numerical experiments concerning convolutional networks we shall make use of average pooling layers.
	\end{remark}
	
	\subsection{Residual blocks}\label{ssec:residual_blocks}
	
	Residual blocks are a group of layers of great importance in machine learning models. This type of blocks induce benefits on both forward and backward propagation \cite{he2016identity}. They are employed, for example, in classification tasks \cite{Deng09} and transformers models including large language model (LLM) as BERT \cite{Tavan21} and GPT models \cite{radford18} or in generative models like the AlphaFold system, employed to predict protein structures \cite{Torr20}. In our framework, we describe the skipping action of the residual blocks via a map acting on the product of two manifolds: the first one retain the data in input of the residual block, whilst the second one follows the forward propagation of the network. This map, realizing the skip connection, can be implemented using a Cartesian product of layers maps and the identity function, defined as follows. Given two functions $f: M_f \rightarrow N_f$ and $g: M_g \rightarrow N_g$ we define their Cartesian product $f \times g : M_f \times M_g \rightarrow N_f \times N_g$ as the map sending $(x,y) \in M_f \times M_g$ to $(f(x),g(y)) \in N_f \times N_g$. Suppose that  there is a residual block skipping $n$ connection whose the input manifold $M_0$ and its output manifold is $M_n$. Then we can create a copy of the input data substituting to the manifolds $M_i$, $i=0,\ldots,n$ the Cartesian product $M_i \times M_0$ and to the layer maps the products $\Lambda_i \times Id$. In this way, the first factor encompasses the layers inside the block, while the identity map retains the information from the input manifold, implementing hence the skipping connection. At the end of the block, under the hypothesis $dim(M_0)=dim(M_n)$, we sum the factors in the output of $\Lambda_n \times Id$, namely $M_n \times M_0$. Note that, in order to give meaning to the sum operation, $M_0$ and $M_n$ must be at least manifolds of the same dimension for which a notion of addition is defined. In practice, both $M_0$ and $M_n$ are manifolds realized as subspaces of the same Euclidean space, therefore the sum of two elements is well-defined.
	
	\begin{remark}
		This strategy, \emph{i.e.}, using the Cartesian product of manifolds and the identity map, can be easily generalized to U-Net architecture with skip connections \cite{benfenati2023neural,ulyanov2018deep}.
	\end{remark}
	
	\begin{figure}[h!]
		\begin{center}
			\includegraphics[height=0.25\textheight]{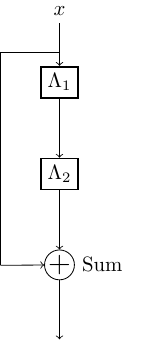}\includegraphics[height=0.25\textheight]{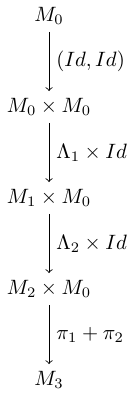}
		\end{center}
		
		\caption{Comparison between the usual visual representation of a residual block with our formalization. Left panel: The input $x$ of the layer $\Lambda_1$ is also merged with the output of the layer $\Lambda_2$, skipping the connections via an identity mapping. Right panel: In our framework we can formalize the skipping connection using product manifolds and tensor products of the layer maps with the identity function. For example the action of $\Lambda_2 \times Id$ on $(x,\Lambda_1(x)) \in M_1 \times Id$ is given by $\Lambda_2 \times Id (x,\Lambda_1(x)) = (\Lambda_2 \circ \Lambda_1(x), x) \in M_2 \times M_0$. The maps $\pi_1$ and $\pi_2$ are the projections on the first and second factor respectively, namely if $(p,q) \in M_2 \times M_0$ then $\pi_1(p,q)=p \in M_2$ and $p_2(p,q)=q \in M_0$.}
		\label{fig:residual}
	\end{figure}
	We formalize the above discussion in the following definition.
	\begin{definition}
		A $\mathcal{C}^1$ residual block of length $k$ starting at $M_i$ and ending at $M_{i+k}$ is a sequence of maps between manifolds of length $k$ of the form 
		\begin{equation*}\label{eq:residual_block}
		\begin{tikzcd}[column sep=large]
		M_{i} \arrow[r, "\textit{(Id,Id)}"] & M_{i} \times M_i \arrow[r, "\Lambda_i \times \textit{Id}"] & M_{i+1} \times M_i \arrow[r,"\Lambda_{i+1} \times \textit{Id}"]  & \cdots \arrow[r,"\Lambda_{i+k-1} \times \textit{Id}"] & M_{i+k-1} \times M_i \arrow[r, "\pi_1 + \pi_2"] & M_{i+k}
		\end{tikzcd}
		\end{equation*}
		where all the maps $\Lambda_i,\ldots,\Lambda_{i+k-1}$ are $\mathcal{C}^1$ and $M_i$, $M_{i+k-1}$ are manifolds of the same dimension realized as subsets of a Euclidean space.
	\end{definition}
	Residual networks fall into the singular Riemannian geometry framework depicted in \cref{sec:geometric_framework}, all the Assumptions are met. Moreover, since the maps $\Lambda_i$ are $\mathcal{C}^1$, for each layer we have a composition of function which is $\mathcal{C}^1$: this means the hypothesis of \cref{prop:smooth} and of \cref{prop:c_1} are satisfied and their results can be applied. The theoretical assumptions for applying \cref{al:SIMECnD} and \cref{al:SIMEXPnD} are met, hence the construction of equivalence class \cite{BeMa21b} can be pursued for this class of networks.

	\subsection{Recurrent layers and recurrent networks}\label{ssec:recurrent_layers_and_networks}
	Finally we focus on recurrent layers and networks, commonly employed in the analysis of time series \cite{Jaeger2004}, handwriting recognition \cite{NIPS2008}, machine translation \cite{sutskever2014}, speech recognition \cite{Santiago2007,Hasim14,Li2015} and text-to-speech synthesis \cite{Fan2015}. Compared to the layers and blocks we considered in the previous sections, recurrent layers exhibit temporal dynamic behaviour. A common way to understand the action of recurrent neural networks is to unfold them in time.
	\begin{figure}
		\begin{center}
			\includegraphics[scale=1]{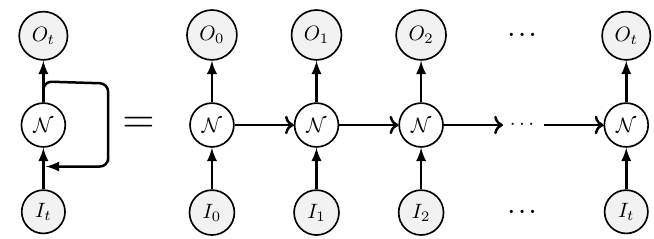}
		\end{center}
		\caption{A recurrent network (left) and its unrolling in time (right). In addition to the input datum $I_t$ at time $t$, also the output produced at the prevous time $O_{t-1}$ is given as input to the network.}
		\label{fig:recurrent_network_unrolled}
	\end{figure}
	\begin{definition}[Recurrent neural network]\label{def:recurrent_network}
		A Fully recurrent neural networks (FRNN) is a sequence of pairs $(\mathcal{N},r_t)$, $r \in \mathbb{N}$, such that:
		\begin{itemize}
			\item[1)] The input manifold $M_0$ is a product manifold of the form $U_0 \times R_0$, where $U_0$ is the current input data manifold and $R_0$ represents a memory manifold, storing the data from the previous time step.
			\item[2)] $r_t \in R_0$ for every $t \in \mathbb{N}$, with $r_0 = 0$.
			\item[3)] For a given input $u \in U_0$
			\begin{equation}
			\begin{cases}
			r_0 = 0, \\
			r_{t} = \pi_2 \mathcal{N}((u,r_{t-1}))
			\end{cases}
			\end{equation}
			where $\pi_2$ is the projection on the second factor of $M_0 = U_0 \times R_0$.
		\end{itemize}
		We call a pair $(\mathcal{N},r_t)$ at a certain time $t$ the state of the network at time $t$.
	\end{definition}
	\begin{figure}[h!]
		\begin{equation}
		\begin{tikzcd}
		\mathcal{O}_0 = R_{1}  & \mathcal{O}_1 = R_{2} & \cdots & \mathcal{O}_i =  R_{i+1}\\
		M_{0} \times R_{0} \arrow[u, "\mathcal{N}"]  & M_{1} \times R_{1} \arrow[u, "\mathcal{N}"] & \cdots & M_{i} \times R_{i} \arrow[u, "\mathcal{N}"] \\
		o_0 = r_{1}  & o_1 = r_{2} & \cdots & o_i =  r_{i+1}\\
		(u_{0}, r_{0}) \arrow[u, mapsto] & (u_{1}, r_{1}) \arrow[u, mapsto] & \cdots & (u_i,r_i) \arrow[u, mapsto] \\
		\textit{State of } \mathcal{N} \textit{ at } t=0 & \textit{State of } \mathcal{N} \textit{ at } t=1 & \cdots & \textit{State of } \mathcal{N} \textit{ at } t=i\\
		\end{tikzcd}
		\end{equation}
		\caption{Calling $\mathcal{O}_t$ the output manifold at time $t$ and $o_t$ the output corresponding to the input $(u_t,r_t)$, the unrolling of the networks in time yields the sequence represented above, as per \Cref{def:recurrent_network}. At each time the function $\mathcal{N}$ maps the pair $(u_t,r_t)$ -- the input at time $i$ and the previous output of the network respectively -- to $o_t = \mathcal{N}((u_t,r_t))$, which then will be employed also in the next timestep. Compare this representation to that given in \Cref{fig:recurrent_network_unrolled}.}
	\end{figure}
	Recurrent layers can be treated in a similar fashion, unrolling them in time. 
	\begin{definition}[Recurrent layer]\label{def:recurrent_layer}
		A recurrent layer is a sequence of pairs $(\Lambda,r_t)$, $r \in \mathbb{N}$, such that:
		\begin{itemize}
			\item[1)] The input manifold $M$ of is a product manifold of the form $U \times R$, where $U$ is the current input data manifold and $R$ represents a memory manifold, storing the data from the previous time step.
			\item[2)] $r_t \in R$ for every $t \in \mathbb{N}$, with $r_0 = 0$.
			\item[3)] For a given input $u \in U$
			\begin{equation}
			\begin{cases}
			r_0 = 0, \\
			r_{t} = \pi_2 \mathcal{N}((u,r_{t-1}))
			\end{cases}
			\end{equation}
			where $\pi_2$ is the projection on the second factor of $M = U \times R$.
		\end{itemize}
		We call a pair $(\Lambda,r_t)$ at a certain time $t$ the state of the layer at time $t$.
	\end{definition}
	From the formalisation of recurrent layer just given, making use of the notion of Cartesian product to store the previous state introduced in \cref{ssec:residual_blocks}, it is clear that -- as in the case of residual blocks -- such kind of networks abide to the framework of \cref{sec:geometric_framework}, and the results from \cref{sec:c_k} are valid.
	\begin{remark}
		An immediate consequence of the definition above is that on a recurrent neural network the equivalence classes on $M_0$ and on the representation manifolds $M_i$ change in time, since the transformation applied by the maps $\Lambda_i$ are depending on the previous state of the network. In other words, each state of the network has its own equivalence classes.
	\end{remark}
	
	For example, let us consider a long short-term memory (LSTM) units, a kind of block composed by a cell remembering previous values over arbitrary time intervals and three gates, often called forget, update and output gates \cite{Hochreiter97}.
	\begin{figure}[h!]
		\centering
		\includegraphics[scale=1]{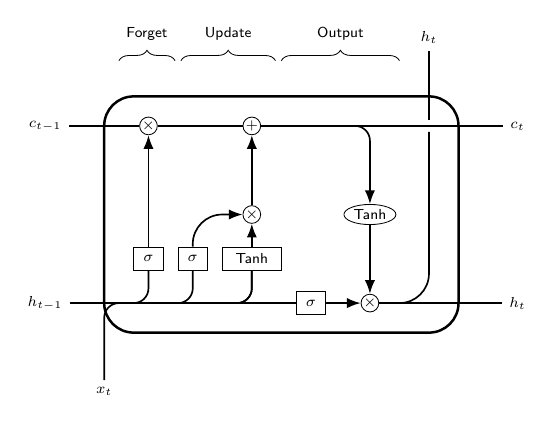}
		\caption{A LSTM unit. $x_t$ is the input of the LSTM unit, $h_t$ is both the current output and the next value of the hidden state, $h_{t-1}$ is the current hidden state -- namely the previous output. $c_{t-1}$ and $c_t$ are the memory cell internal states, allowing the cell to remember. Roughly speaking, the role of the three gates is the following. First, the input gate determines how much of the current input should be added to the current memory cell internal state. The data is then passed on to the forget gate, which decides whether to keep the current value of the memory or flush it. At last the output gate determines how much the memory cell should influence the output at the current time step.}
		\label{fig:lstm}
	\end{figure}
	Referring to \Cref{def:recurrent_layer}, we identify the pair $(c_{t-1},h_{t-1})$ as the data carried on from the previous state, lying in the space $R_t$, while the input manifold is the space to which the input $x_t$ belongs. We can see from the visual inspection given in \Cref{fig:lstm} that a LSTM unit performs only smooth operations on data, namely adding and multiplying data, applying a sigmoid or a hyperbolic tangent function and copying data to be employed for the next step -- which corresponds to make use of the identity function from $\mathbb{R} \rightarrow \mathbb{R} \times \mathbb{R}$. 
	
	\section{Non-differentiable activation functions}\label{sec:extension_non_differentiable}
	\subsection{General considerations}\label{subsec:continuous_1}
	In the last few years several authors studied the inner working of ReLU neural networks focusing on the arrangement of the activation hyperplanes and on the subsequent folding of the data manifold due to nonlinear activation function to understand the transformation implemented by the network \cite{Raghu16,Balestriero19,Hanin19,Rolnick20,he2021relu,Black22,Keup22}.
	Following these papers we study the geometry of the input manifold applying our geometric framework inside the polytopes individuated by the activation hyperplanes. Eventually, the methods we develop can be applied also to the representation manifolds of the inner layers. We extend the results of \Cref{sec:geometric_framework} to some useful cases where the activation functions are merely continuous or presenting jump discontinuities, ReLU being the key case. Functions in the latter case should not be too pathological. By `non pathological' we mean a function which is differentiable almost everywhere, except for a small subset of the domain, in a sense we will specify later. An extreme example of pathological function is given by the Dirichlet function 
	$$
	\chi(x) = 
	\begin{cases}
	1 \textit{ if } x \in \mathbb{Q}\\
	0 \textit{ if } x \in \mathbb{R}\setminus\mathbb{Q}
	\end{cases} 
	$$
	Since $\chi$ is nowhere continuous, in particular it is nowhere differentiable. This means that we cannot even compute the pullback $\chi^* g$ using \Cref{eq:pullback}. On the other hand the ReLU function $\sigma(x)=\max\{0,x\}$ is not differentiable just at $x=0$. Its derivative, a step function, is continuous but for $x=0$, therefore in every point other than $x=0$ the pullback $\sigma^* g$ can be computed and is a continuous function. This reasoning remains true also for the composition of $\sigma$ with a $1 \times n$ matrix $A$, with the difference that the subset of $\mathbb{R}^n$ for which $\sigma \circ A$ is not differentiable is a $n$-dimensional hyperplane passing through the origin. Supported by these examples, we make the following assumption.
	\begin{assumption}
		The activation functions considered hereafter belong to the class ae--$\mathcal{C}^1(\Omega)$, $\Omega\subset\mathbb{R}^n$.
	\end{assumption}
	Except for a set with zero  measure, the pullback can be computed and therefore we can build the null curves as in the previous section. It remains to study the relation between null curves and equivalence classes, which can be influenced by the presence of a null set of points where $f$ is not differentiable. The case of functions with a jump discontinuity is treated in \Cref{sec:jump}, where we treat a particular case satisfying additional assumptions.  We note that the results of \cite{BeMa21a} does not apply globally. In particular the hypotheses of Godement's theorem are not satisfied and we hence cannot conclude that all the class of equivalence are submanifolds of the same dimension, which can vary from a class to another one. Even some simple functions exhibit this behavior.
	\begin{example}[ReLU layer]\label{ex:relu_1}
		Consider a function $f: \mathbb{R}^2 \rightarrow \mathbb{R}$ defined as $f:= \relu \circ A$, with 
		\begin{equation*}
		A= 
		\begin{pmatrix}
		3 \\
		0 \\
		0
		\end{pmatrix}^T
		\end{equation*}
		We want to study the set of the points which are mapped to the same output, namely the set $\mathbb{R}^3/\sim$ with $\sim$ the equivalence relation defined by $q \sim p$ if and only if $f(q)=f(p)$, $p,q \in \mathbb{R}^3$. Let $p$ and $q$ be two points of $\mathbb{R}^3$ such that $f(p)=f(q)$. We claim that the class of equivalence in which $q$ and $p$ belongs are either of dimension $3$ or $2$. Suppose that $p,q$ are such that $Ap > 0$ and $Aq >0$. Then $\relu(q)=\relu(p)$ entails that $Aq=Ap$ or, equivalently, that $p-q \in \Ker(A)$. Since $A$ is a non-null $1 \times 3$ matrix, we know that $\dim \Ker(A) = 2$, from which we conclude that the class of equivalence of $q$ and $p$ is a $2$-dimensional space. Consider now two points $p,q$ such that $Ap\leq 0$ and $Aq \leq 0$. By the definition of $\relu$, $\relu(p)=\relu(q)=0$. Therefore, since the octant $Q$ of $\mathbb{R}^3$ with $x,y,z < 0$ is such that every point $q \in Q$ satisfies $\relu(q)=0$ we conclude that $Q$ is a class of equivalence of dimension $3$.
	\end{example}

	However, in each open, path-connected region of $\mathbb{R}^n$ where the function is differentiable the results of \cite{BeMa21a} and \Cref{sec:c_k} hold true. Indeed, suppose that $f$ is differentiable in some open and path-connected subsets $M_0,M_1,\ldots \subset \mathbb{R}^n$. Then we can define the $\mathcal{C}^1$ maps $f_0,f_1,\ldots$ which are the restrictions of $f$ to $M_0,M_1,\ldots$. Each of these functions lies in the case treated in \Cref{sec:c_k}. In other words, we can apply the results of \Cref{sec:c_k} on each one of these regions, which can be foliated by classes of equivalence possibly of different dimension. We know from \Cref{sec:c_k} that two point in the same class of equivalence are connected by a null curve. The crucial point is to study the converse statement, namely if a null curve starting in one of the region $M_i$ remains in $M_i$. This property can be studied analyzing the rank of the pullback metric $f^*g$.
	\begin{example}[Pullback metric in a ReLU layer]\label{ex:relu_2}
		Consider again the map of \Cref{ex:relu_1}. The pullback of the standard metric of $\mathbb{R}$ with respect to $f$ is
		\begin{equation}
		f^*g = 
		\begin{pmatrix}
		\left(3\relu^\prime\right)^2 & 0 & 0 \\
		0 & 0 & 0 \\
		0 & 0 & 0 \\
		\end{pmatrix}
		\end{equation}
		Notice that a null curve whose three components are positive, is such that the pullback metric is of rank 1, since $\relu^\prime = 1$ in this region. A null curve whose components are negative is such that $\relu^\prime = 0$ and therefore the rank of $f^*g$ is zero. In general, there may exist piecewise $\mathcal{C}^1$ null curves starting in the first region -- corresponding to equivalence classes of dimension $2$ -- and ending in the second one, part of an equivalence class of dimension $3$. Therefore, unlike the $\mathcal{C}^k$ case, null curves alone cannot be employed to reconstruct an equivalence class: In addition we must require that along the curve the pullback metric $f^*g$ is of constant rank. Note that the rank of the pullback metric signals the presence of a point in which $\relu$ is not differentiable, since the sign of $Ap$, $p \in \mathbb{R}^3$, discriminates the two cases discussed above.
	\end{example}
	The general case requires more advanced notions from differential geometry, such as distribution, a kind of generalized function for which the pullback is not always well-defined \cite{Hor83,Mel03}: we limit ourselves to study some commonly employed nonlinear maps of practical interest. 
	
	\subsection{Composition of monotone and  linear applications}\label{subsec:monotone}
	We begin focusing on the simplest generic case of practical interest, namely the composition between a monotone activation function $\sigma : \mathbb{R} \rightarrow \mathbb{R}$ and a linear application represented by a $1 \times n$ matrix, generalizing the example above.
	\begin{proposition}\label{prop:1d-monotone-funct}
		Let $\sigma : \mathbb{R} \rightarrow \mathbb{R}$ be a weakly monotone $\mathcal{C}^0$ function and let $A$ be a $1 \times n$ matrix. Then the class of equivalence of $f = \sigma \circ A $ are either of dimension $n$ or $n-1$.
	\end{proposition}
	\begin{proof}
		If a continuous function $\sigma$ is weakly monotone, then is either constant or it admits a countable number of points in which it is not differentiable. If $\sigma$ is constant, there is nothing to prove, since every point is mapped to the same result. The space $\mathbb{R}^n$ itself is therefore a class of equivalence of dimension $n$. If $\sigma$ is non-constant, then it may be defined as a piecewise function consisting of increasing, decreasing or constant functions. We focus on the case of increasing or constant functions, the proof for the non-increasing case being the same. Non-decreasing continuous functions can be built in general gluing together constant functions and strictly increasing functions. Let $p \in \mathbb{R}^n$ be a point such that $\sigma^\prime(Ap) \neq 0$. Then, being $\sigma$ monotone, there is a neighborhood $\mathcal{U}_p$ of $p$ for which $\sigma$ is strictly monotone. Consider $q\in\mathcal{U}_p$: therefore $\sigma(Ap) = \sigma(Aq)$ if and only if $Ap=Aq$, namely $q = p +\Ker(A)$, and $\Ker(A)$ is a space of dimension $n-1$. Let now $p \in \mathbb{R}^n$ be a point such that $\sigma^\prime(Ap)=0$. The hypothesis on $\sigma$ entails that there is a neighborhood $U$ of $y=Ap$ in which $\sigma^\prime$ is null. The results of \cite{BeMa21a} entails that the class of equivalence is a submanifold of $\mathbb{R}^n$ of dimension $n$.
	\end{proof}
	An immediate consequence is the following property.
	\begin{corollary}\label{cor:same_class}
		Let $f=\sigma \circ A$ be as in \Cref{prop:1d-monotone-funct}. Then two points in the same class of equivalence are connected by a null curve.
	\end{corollary}
	The converse of this statement is not necessarily true. Consider a function $\sigma$ of the form 
	$$\sigma(x) =
	\begin{cases}
	-x & \textit{  if } x \leq 0\\
	2x & \textit{ if } x > 0
	\end{cases}
	$$
	Since $\sigma(x)\neq0$ for every $x \neq 0$, the rank of the pullback is never null, therefore we cannot detect the point in which the function is not differentiable using the pullback metric, in contrast with what happens in \Cref{ex:relu_2}. However, there is a class of activation functions generalizing the ReLU example for which the non differentiability of $\sigma$ is reflected to a change of rank of the pullback metric.
	\begin{proposition}\label{prop:null_continuous_1}
		Let $f=\sigma \circ A$ be as in \Cref{prop:1d-monotone-funct} and assume that $\sigma$ is either of the form
		\begin{equation}\tag{a}
		\sigma(x) =
		\begin{cases}
		0 & \textit{ if } -\infty < x \leq a\\
		h(x) & \textit{ if } x > a
		\end{cases}
		\end{equation}
		with $a \in \mathbb{R}$ and $h(x)$ a differentiable, strictly increasing function such that $\displaystyle \lim_{x\to a^+}h(x)=0$ and $\displaystyle \lim_{x\to a^+}h^\prime(x)=\ell \neq 0$, for some $\ell \in \mathbb{R}$, or of the form 
		Let $\gamma:[a,b]\rightarrow\mathbb{R}^n$ be a null curve. Then two point $p,q$ belonging to different equivalence classes cannot belong to $Im(\gamma)$. The same result holds when $h$ is a strictly decreasing function.
	\end{proposition}
	\begin{proof}
		The statement is a consequence of the results of \Cref{sec:c_k} applied the two regions individuated in \Cref{prop:1d-monotone-funct}, see also the discussion below \Cref{ex:relu_1}.\end{proof}
	\begin{remark}\label{rem:null_curve_and_eq_classes}
		This proposition plays a key role from a numerical point of view. Indeed, when we build a null curve integrating the system
		\begin{equation*}
		\begin{cases}
		\dot{\gamma}(s) = V(x)\\
		\gamma(0) = p
		\end{cases}
		\end{equation*}
		where $p \in \mathbb{R}^n$ the starting point of the null curve and $V(x)$ a vector field in $\Ker(f^*g)$, it may happen that, using any numerical integration algorithm, a null curve starting in an equivalence class of dimension $n-1$, sooner or later pass to another one of dimension $n$. To avoid this scenario, it is sufficient to introduce in the algorithms an additional step checking the dimension of $\Ker(f^*g)$. We also note that in order to select a null eigenvector, we need to know the number of null eigenvectors, thus we already computed $\dim \Ker(f^*g)$. 
	\end{remark}
	We note that the above proposition, encompassing the ReLU case, does not apply to other commonly employed activation functions, such as leaky ReLU.
	\begin{example}[Pullback metric in a leaky ReLU layer]\label{example:leaky_pullback}
		Consider a function $f:\mathbb{R}^3\to\mathbb{R}$, $f=\sigma \circ A$, with
		\begin{equation*}
		\sigma(x) = 
		\begin{cases}
		-0.1 x & \textit{ if } x < 0\\
		x & \textit{ if } x \geq 0
		\end{cases},\quad A = 
		\begin{pmatrix}
		1 \\
		1 \\
		1 \\
		\end{pmatrix}
		\end{equation*}
		Then the pullback metric is given by
		\begin{equation*}
		f^*g = 
		\left( \sigma^\prime \right)^2 
		\begin{pmatrix}
		1 & 1 & 1 \\
		1 & 1 & 1 \\
		1 & 1 & 1
		\end{pmatrix}
		\end{equation*}
		Since $\sigma^\prime$ is never zero, the rank of $f^*g$ is always one. In this case the degeneracy of the metric is not detecting the point in which $\sigma$ is not differentiable. As a consequence, from a numerical point of view, checking the dimension of $\Ker(f^*g)$ is not enough to make sure the a null curve does not pass from an equivalence class to another one.
	\end{example}
	The example shows that in the case of leaky ReLU the singular metric alone is not enough to identity an equivalence class. To extend the previous proposition for a class of functions generalizing leaky ReLU, we need an additional hypothesis.
	\begin{proposition}\label{prop:null_continuous_2}
		Let $f=\sigma \circ A$ be as in \Cref{prop:1d-monotone-funct} and assume that $\sigma$ is a continuous piecewise function built gluing $k$ linear functions, namely that $f$ is of the form
		\begin{equation*}
		f(x) =
		\begin{cases}
		m_1 x + q_1 \textit{ if } -\infty \leq x < a_1\\
		m_2 x + q_2 \textit{ if } a_1 \leq x < a_2\\
		\quad \quad \quad \quad \quad \vdots\\
		m_{k-1} x + q_{k-1} \textit{ if } a_{k-1} \leq x < a_{k-1}\\
		m_{k} x + q_{k} \textit{ if } a_{k} \leq x \leq \infty
		\end{cases}
		\end{equation*}
		where $m_i,q_i \in \mathbb{R}$  with $m_i \neq m_{i+1}$ for every $i=1,\ldots,k-1$ and such that $\lim_{x\rightarrow a_i^-}f(x) = f(a_i)$ for every $i=1,\ldots,k$. Let $\gamma:[a,b]\rightarrow\mathbb{R}^n$ be a null curve such that $f^\prime(\gamma(s))$ is constant $\forall s \in [a,b]$. Then two point $p,q$ belonging to different equivalence classes cannot belong to $Im(\gamma)$.
	\end{proposition}
	\begin{proof}
		Since $\sigma^\prime$ is never null, form the proof of \Cref{prop:1d-monotone-funct} we know that the equivalence classes are of dimension $n-1$. As a consequence of \Cref{prop:null_continuous_1} the pullback metric alone, of constant rank, is not detecting the points in which $\sigma$ is not differentiable, thus null curves may pass from an equivalence class to another one. The hypoteses on $\sigma$ precludes this possibility, since different equivalence classes correspond to different angular coefficients of $\sigma$: consider $x$ and $y$ such that $\sigma^\prime(x)=m_l, \, \sigma^\prime(y)=m_k$, then $[x]\neq[y]$.
	\end{proof}
	\begin{remark}\label{remark:discontinuity_of_the_metric}
		Observe that, in agreement with the previous proposition, the pullback metric computed in \Cref{example:leaky_pullback} detects the presence of two equivalence classes with a discontinuity of $f^* g$. Indeed, let $q=(q_1,q_2,q_3) \in \mathbb{R}^3$. Then the input of the function $\sigma$ is given by $Aq=q_1+q_2+q_3$. If $Aq<0$, then $\sigma^\prime$ is $-0.1$ otherwise is $1$, therefore
		\begin{equation*}
		f^*g(Aq < 0) = 
		\begin{pmatrix}
		0.01 & 0.01 & 0.01 \\
		0.01 & 0.01 & 0.01 \\
		0.01 & 0.01 & 0.01
		\end{pmatrix}
		\quad
		\textit{while}
		\quad
		f^*g(Aq \geq 0) = 
		\begin{pmatrix}
		1 & 1 & 1 \\
		1 & 1 & 1 \\
		1 & 1 & 1
		\end{pmatrix}
		\end{equation*}
	\end{remark}
	
	At last, we note that \Cref{cor:same_class}, along with \Cref{prop:null_continuous_1} and \Cref{prop:null_continuous_2}, yields a generalization of \cite[Proposition 4]{BeMa21a} to the case of a layer defined making use of non differentiable functions abiding to the hypotheses of \Cref{prop:null_continuous_1,prop:null_continuous_2}.
	
	\begin{proposition}\label{prop:one_dimensional_non_smooth_layer}
		Let $f = \sigma \circ A$, with $A$ a $1 \times n$ matrix and $\sigma : \mathbb{R} \rightarrow \mathbb{R}$ a weakly monotone $\mathcal{C}^0$ function satisfying the hypotheses of either \Cref{prop:null_continuous_1} or \Cref{prop:null_continuous_2}. Let $p,q \in \mathbb{R}^n$. Then $f(p) = f(q)$ if and only if $p$ and $q$ belong to the same equivalence class.
	\end{proposition}
	
	\subsection{Layers of generic dimension}\label{subsec:relu_layers}
	The results of the previous section can be applied only on non-smooth layers whose output is of dimension one. Now we relax this hypothesis. In principle we may proceed applying \cref{prop:1d-monotone-funct}, \cref{prop:null_continuous_1}, \cref{prop:null_continuous_2}, \cref{prop:one_dimensional_non_smooth_layer} to each component of the map $\Lambda$ realizing a layer and then intersecting the equivalence classes of the input space we obtained from each component. However \Cref{remark:discontinuity_of_the_metric}, along with the fact that we can apply the machinery of \cite{BeMa21a,BeMa21b} in each region in which the pullback metric is defined, suggests to follow a different path -- We can take exploit the discontinuity of the metric to detect the presence of an activation hyperplane. Before tackling the problem in full generality, we illustrate this strategy using a low-dimensional ReLU layer.
	
	Let $F: \mathbb{R}^3 \rightarrow \mathbb{R}^3$ defined as $F(x,y,z) = (\relu(x),\relu(y),\relu(z))$, and let $A$ be a $3 \times 2$ full rank matrix.
	Consider the layer $\Lambda = F \circ A : \mathbb{R}^2 \rightarrow \mathbb{R}^3$
	\begin{equation}
	(x,y) \mapsto
	\begin{pmatrix}
	\relu(a_{11}x+a_{12}y+b_1)\\
	\relu(a_{21}x+a_{22}y+b_2)\\
	\relu(a_{31}x+a_{32}y+b_3)\\
	\end{pmatrix}
	\end{equation}
	where the biases $b_1,b_2,b_3$ are real numbers.
	Each component of $\Lambda$ satisfies the hypothesis of \Cref{prop:1d-monotone-funct}. In particular the gradient of the $j$-th component $\Lambda^j$ is
	\begin{equation}
	\nabla \Lambda^j =
	\begin{cases}
	(0,0) & \quad \textit{if} \quad a_{j1}x+a_{j2}y+b_j \leq 0\\
	(a_{j1},a_{j2}) & \quad \textit{if} \quad a_{j1}x+a_{j2}y+b_j > 0\\
	\end{cases}
	\end{equation}
	Above the line $a_{j1}x+a_{j2}y+b_j$ the gradient of $\Lambda^j$ is $(a_{j1},a_{j2})$ while below the line is $(0,0)$, therefore we individuate up to $7$ regions in $\mathbb{R}^2$ encompassing all the different combinations for the gradients making up the Jacobian of the layer $\Lambda$.
	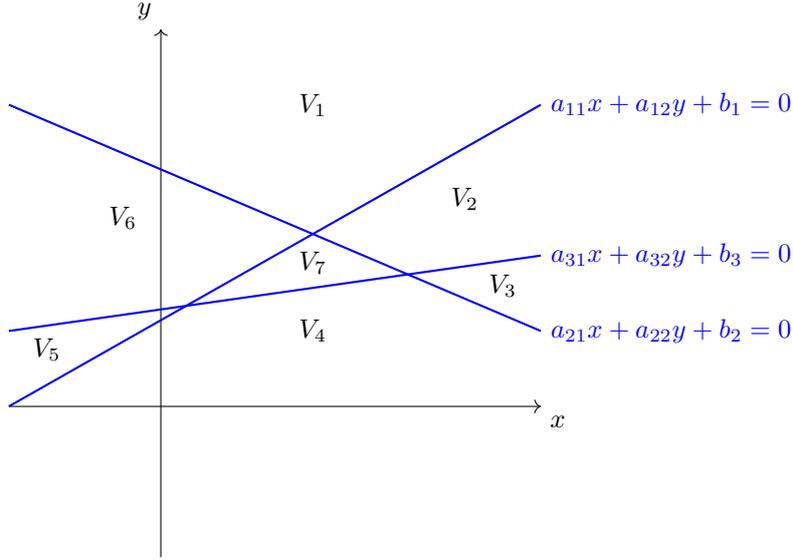
\begin{figure}[htbp]
		\begin{center}
			\begin{tikzpicture}
			\draw[->] (-2,0) -- (5,0) node[anchor=north west] {$x$};
			\draw[->] (0,-2) -- (0,5) node[anchor=south east] {$y$};
			\draw[blue,thick] (-2,0) -- (5,4)  node[anchor=west] {$a_{11}x+a_{12}y+b_1 = 0$};
			\draw[blue,thick] (-2,4) -- (5,1)  node[anchor=west] {$a_{21}x+a_{22}y+b_2 = 0$};
			\draw[blue,thick] (-2,1) -- (5,2)  node[anchor=west] {$a_{31}x+a_{32}y+b_3 = 0$};
			\node[thick] at (2,4) {$V_1$};
			\node[thick] at (4,2.75) {$V_2$};
			\node[thick] at (4.5,1.6) {$V_3$};
			\node[thick] at (2,1) {$V_4$};
			\node[thick] at (-1.5,0.75) {$V_5$};
			\node[thick] at (-0.5,2.5) {$V_6$};
			\node[thick] at (2,1.9) {$V_7$};
			\end{tikzpicture}
		\end{center}
		\caption{Visual inspection of the regions individuated by a ReLU layer coupled with a $3\times2$ matrix.}
		\label{fig:regions}
	\end{figure}
	Denoting with $J\Lambda (V_i)$ the the Jacobian of $\Lambda$ in the seven regions, $i=1,\ldots,7$ we readily find
	\begin{equation*}
	J\Lambda (V_1) =
	\begin{pmatrix}
	a_{11} & a_{12}\\
	a_{21} & a_{22}\\
	a_{31} & a_{32}\\
	\end{pmatrix}
	\quad
	J\Lambda (V_2)=
	\begin{pmatrix}
	0 & 0\\
	a_{21} & a_{22}\\
	a_{31} & a_{32}\\
	\end{pmatrix}
	\quad
	J\Lambda (V_3)=
	\begin{pmatrix}
	0 & 0\\
	a_{21} & a_{22}\\
	0 & 0\\
	\end{pmatrix}
	\end{equation*}
	\begin{equation*}
	J\Lambda (V_4) =
	\begin{pmatrix}
	0 & 0\\
	0 & 0\\
	0 & 0\\
	\end{pmatrix}
	\quad
	J\Lambda (V_5)=
	\begin{pmatrix}
	a_{11} & a_{12}\\
	0 & 0\\
	0 & 0\\
	\end{pmatrix}
	\quad
	J\Lambda (V_6)=
	\begin{pmatrix}
	a_{11} & a_{12}\\
	a_{21} & a_{22}\\
	0 & 0\\
	\end{pmatrix}
	\end{equation*}
	\begin{equation*}
	J\Lambda (V_7) =
	\begin{pmatrix}
	0 & 0\\
	0 & 0\\
	a_{31} & a_{32}\\
	\end{pmatrix}
	\end{equation*}
	We can distinguish the regions looking at which gradients are null in the Jacobian matrix. In the interior of each region $\Lambda$ is smooth, therefore we can apply the results of \cite{BeMa21a,BeMa21b}. 
	Suppose now to consider $\mathbb{R}^3$ endowed with its Euclidean metric $g=diag(1,1,1)$. Computing the pullback through $\Lambda$ in the different regions yields
	\begin{equation*}
	h(V_1) =
	\begin{pmatrix}
	a_{11}^2+a_{21}^2+a_{31}^2 & a_{11} a_{12}+a_{21} a_{22}+ a_{31}a_{32} \\
	a_{11} a_{12}+a_{21} a_{22}+ a_{31}a_{32} & a_{12}^2+a_{22}^2+a_{32}^2
	\end{pmatrix}
	\end{equation*}
	\begin{equation*}
	h(V_2) =
	\begin{pmatrix}
	a_{21}^2+a_{22}^2 & a_{21} a_{22}+a_{31} a_{32} \\
	a_{21} a_{22}+a_{31} a_{32}  & a_{31}^2+a_{32}^2
	\end{pmatrix}
	\end{equation*}
	\begin{equation*}
	h(V_3) =
	\begin{pmatrix}
	a_{21}^2 & a_{21} a_{22} \\
	a_{21} a_{22}  & a_{22}^2
	\end{pmatrix}
	\quad
	h(V_4) =
	\begin{pmatrix}
	0 & 0 \\
	0 & 0
	\end{pmatrix}
	\quad
	h(V_5) =
	\begin{pmatrix}
	a_{11}^2 & a_{11} a_{12} \\
	a_{11} a_{12}  & a_{12}^2
	\end{pmatrix}
	\end{equation*}
	\begin{equation*}
	h(V_6) =
	\begin{pmatrix}
	a_{11}^2+a_{12}^2 & a_{11} a_{12}+a_{21} a_{22} \\
	a_{11} a_{12}+a_{21} a_{22}  & a_{21}^2+a_{22}^2
	\end{pmatrix}
	\quad
	h(V_7) =
	\begin{pmatrix}
	a_{31}^2 & a_{31} a_{32} \\
	a_{31} a_{32}  & a_{32}^2
	\end{pmatrix}
	\end{equation*}
	Note that in the pullback metric the transition from a region to another may yield a discontinuity of at least one of the entries, for example going from region $V_3$ to region $V_4$ makes the first entry of the first row go from $a_{21}^2$ to $0$. In general, since the composition of discontinuous functions is not necessarily discontinuous, the pullback metric may be continuous also on some of the activation lines, allowing us to join two regions and to consider the pullback metric defined on their union. However, \Cref{prop:one_dimensional_non_smooth_layer} is precluding that points of the same equivalence class can be in two (or more) different regions, as we will discuss in the proof of the next result. In the example above, where we employed the ReLU function, this means that we cannot find any weight matrix $A$ for which the pullback metric is the same in two regions sharing a boundary, i.e. it cannot happen that $h(V_4)=h(V_7)$ -- otherwise we could find a null curve starting in $V_4$ and ending in $V_7$, in violation of \Cref{prop:one_dimensional_non_smooth_layer}. Now we prove the following general proposition, encompassing ReLU activation functions with input and output spaces of any dimensions.
	\begin{proposition}\label{prop:relu_equivalence_classes}
		Let $A$ be a $m \times n$ matrix, $m,n \in \mathbb{N}$ and let $\sigma : \mathbb{R}^n \rightarrow \mathbb{R}^n$ be a function whose components satisfy the hypotheses of \Cref{prop:null_continuous_1}. Set $\Lambda = \sigma \circ A$. Then the pullback of a (degenerate) Riemannian metric $g$ through $\Lambda$ induces equivalence class of dimensions at most $n$ in the input space $\mathbb{R}^m$ removed of a set of null measure on which $\Lambda$ is not differentiable, where the pullback metric $\Lambda^* g$ is discontinuous.
	\end{proposition}
	\begin{proof}
		The $n$ rows of the linear map $A$ individuate at most $m$ hyperplanes in $\mathbb{R}^m$, with exactly $m$ hyperplanes when $A$ is of full rank. The intersections of these hyperplanes form at most $\mathcal{I} = \sum_{i=0}^m {n \choose i}$ (eventually unbounded) polytopes $\mathcal{P}_i$ in $\mathbb{R}^m$. Since by hypothesis $\sigma$ is continuous on $\mathbb{R}^m$ and differentiable everywhere except for the points of the hyperplanes, inside each polytope we can compute the pullback of the metric $g$. On the hyperplanes $J\Lambda$ is discontinuous, therefore $\Lambda^* g$ may not be defined on these points. Suppose first that $\Lambda^* g$ is not defined on every hyperplane. In this case any equivalence class is completely contained in one and only one polytope -- since we cannot cross the hyperplanes when constructing null curves, two points in two different polytopes cannot be equivalent. If $\Lambda^* g$ is not identifying points inside the polytopes, then each region $\mathcal{P}_i$ consists of a class of equivalence. Otherwise, if $\Lambda^* g$ is degenerate, some of the $\mathcal{I}$ polytopes can be foliated by their own equivalence classes, whose dimension can vary from $1$ to $n$ depending on the rank of the pullback metric. We conclude that the input space $\mathbb{R}^m$ can be decomposed as the union of the hyperplanes, a null set on which $\Lambda^* g$ is discontinuous, and the equivalence classes induced by the metric, eventually of different dimensions if the rank of $\Lambda^* g$ is not constant. It remains to show that even if $\Lambda^* g$ is defined and continuous on some of the hyperplanes, then an equivalence class cannot cross them, reducing the core of the proof to the previous case. Consider two polytopes $\mathcal{P}_1$ and $\mathcal{P}_2$ such that their common boundary $\partial_{12} = \overline{\mathcal{P}_1} \cap \overline{\mathcal{P}_2}$ is a $n$-dimensional subset of one of the activation hyperplanes. Suppose that the pullback metric can be extended continuously to their common boundary $\partial_{12}$. If this was the case we could define a new region $\mathcal{U} = \mathcal{P}_1 \cup \mathcal{P}_2 \cup \partial_{12}$ on which each component of the extended metric $\Lambda^* g$ would be continuous. On the other hand, an equivalence class in $\mathcal{U}$ is also the intersection of the equivalence classes relative to all the components of the map $\Lambda$, for each of which \Cref{prop:one_dimensional_non_smooth_layer} applies. In particular this means that, since $\Lambda(x) \neq \Lambda (y) \ \forall x \in  {\mathcal{P}_1}, \ y \in \mathcal{P}_2$ by definition of activation hyperplane -- changing polytopes at least one component becomes zero -- then a point of $\mathcal{P}_1$ cannot be equivalent to a point in $\mathcal{P}_2$.
	\end{proof}
	\begin{remark} We enlist some remarkable observation on \cref{prop:one_dimensional_non_smooth_layer}:
		\begin{enumerate}
			\item The previous result remains true also considering, instead of a matrix, an affine map of the form $v \mapsto Av + b$, with $A$ a $m \times n$ matrix and $b \in \mathbb{R}^n$.
			\item We note that another trivial consequence of \Cref{prop:one_dimensional_non_smooth_layer} applied componentwise is that all the points in the same equivalence class are mapped to the same output.
			\item In the case of layers employing the ReLU activation function, being the pullback metric constant in each polytope, the reasoning at the end of the previous proposition rules out the possibility that in two consecutive polytopes the metric can be the same. Indeed, if this was the case then we could find equivalence classes lying in two different polytopes, in contrast with \Cref{prop:one_dimensional_non_smooth_layer}. This fact allows to make use of the discontinuity of the metric to detect the activation hyperplanes. \label{remark:jump_discontinuity}
		\end{enumerate}
		
	\end{remark}
	
	The arguments above can be applied also to composition of $n$ non-smooth layers in a network $\mathcal{N}$, since \Cref{prop:relu_equivalence_classes} holds true also for degenerate Riemannian metrics. Indeed we can apply \Cref{prop:relu_equivalence_classes} layer by layer starting from the last layer and proceeding backwards. The first time we compute the pullback of the final Riemannian metric $g_n$, possibily obtaining a degenerate metric $g_{n-1} = \Lambda_n^* g_n$, and we characterize the equivalence classes of the representation space $M_{n-1}$. At each subsequent iteration we compute the pullback of a degenerate metric obtaining the degenerate metrics of the various representation layers, continuing until we get to the input layer and therefore to the metric $g_0$. The result of this procedure is tantamount to consider the pullback metric $\mathcal{N}^*g$ directly. This observation immediately leads to the following proposition.
	\begin{proposition}\label{prop:relu_network_equivalence_classes}
		Let $\mathcal{N}$ be a deep neural network as per \Cref{def:neural_netowork},
		where we assume that $\{\Lambda_i\}$ may also be ReLU layers satisfying the hypotheses in \Cref{prop:null_continuous_1,prop:relu_equivalence_classes}. Let $d_i$ be the dimension of the manifolds $M_i$. Then the pullback of a Riemannian metric $g_n$ through the neural network map $\mathcal{N}$ induces equivalence class of dimensions at most $\max_{i=0,\ldots,n}{d_i}$ in the input space $M_0$ removed of a set of null measure on which $\mathcal{N}$ is not differentiable, where the pullback metric $\mathcal{N}^* g$ is discontinuous.
	\end{proposition}
	At last we note that leaky ReLU activation layers can be treated in a similar fashion, since the activation hyperlanes are the same as the corresponding ReLU ones.
	
	\subsection{Numerical reconstruction of equivalence classes}
	\label{sec:numerical_reconstruction_non_differentiable}
	In this section we adapt the SiMEC algorithm to work with ReLU and leaky ReLU layers. The findings of the previous section, in particular \cref{prop:relu_equivalence_classes,prop:relu_network_equivalence_classes} allow us to employ the SiMEC algorithm as presented in \cite{BeMa21b} ($1D$ version) or in \cref{sec:random_walks_on_equivalence_classes} ($nD$ version) directly, since an equivalence class is entirely contained in a region in which the layers are smooth. However, after a certain number of iterations the numerical errors may lead the reconstructed curve from one polytope to another one. For this reason we propose some improved versions of the algorithm aimed to avoid, or at least to mitigate, these effects. 
	For maps abiding to the hypotheses of \Cref{prop:null_continuous_2}, as a fully connected layer with leaky ReLU activation function, the SiMEC algorithm must be modified as in \Cref{al:SIMEC1D_cont_2}.
	\begin{algorithm}[H]
		\begin{algorithmic}
			\Require {Choose $p_0 \in M_0$, the direction $v_0$, $\delta>0$, maximum number of iteration $K$.}
			\Ensure  {A sequence of points $\{p_s\}_{s=1,\dots,K}$ approximatively in $[p_0]$; The energy $E$ of the approximating polygonal}
			\State{Initialise the energy: $E\gets 0$}
			\State {Compute $g_{\mathcal{N}(p_k)}^n$}
			\State {Compute the pullback metric $g^0_{p_k}$ trough \Cref{eq:pullback}}
			\State {Diagonalize $g^0_{p_k}$ and find the eigenvectors}
			\State{$dimKer \gets$ Number of null eigenvectors}
			\For {$k=1,\dots,K-1$}
			\State {Choose a null eigenvector $w_0$}
			\State {$v_{k} \gets w_0$}
			\If {$v_{k} \cdot v_{k-1}<0$}
			\State {$v_{k} \gets-v_{k}$}
			\EndIf
			\State {Compute the new point $p_{k+1} \leftarrow p_{k}+ \delta v_{k}$}
			\State {Add the contribute of the new segment to the energy $E$ of the polygonal}
			\State {Compute $g_{\mathcal{N}(p_k)}^n$}
			\State {Compute the pullback metric $g^0_{p_k}$ trough \Cref{eq:pullback}}
			\State {Diagonalize $g^0_{p_k}$ and find the eigenvectors}
			\If {Number of null eigenvectors $\neq dimKer$ or $\sigma^\prime(Ap_k) \neq \sigma^\prime(Ap_{k-1})  $}
			\State{\Break}
			\EndIf
			\EndFor
		\end{algorithmic}
		\caption{The $1D$ SiMEC algorithm for leakyReLU-like layers.}\label{al:SIMEC1D_cont_2}
	\end{algorithm}
	We now consider in \cref{al:SIMECnD_modified} the general case treated in \Cref{subsec:relu_layers}. As noted in \eqref{remark:jump_discontinuity} of \cref{remark:jump_discontinuity}, passing through an activation hypersurface yields a jump discontinuity in the pullback metric. Since a jump discontinuity of the pullback metric detects the transition between two different regions, we can make sure to remain in the region containing the equivalence class of the starting point checking the variations of the entries of $h = \Lambda^* g$. From a numerical point of view we can detect this kind of discontinuity making use of a threshold parameter $\tau$ saying that if $|h_{ij}(x)-h_{ij}(y)| > \tau$ then $x$ and $y$ belongs to different regions. If we use a small enough step $\delta = |x-y|$, a suitable choice for the threshold parameter $\tau$ could be $\tau = L \delta$ with $L$ the maximum Lipschitz constant of $\Lambda$ among the different regions. We remark that when using linear operator, such as in Fully Connected layers or convolutional ones, due to linearity they are also bounded, hence one may have a reliable estimation of the constant $L$; indeed, we emphasize that this framework considers trained networks, whose weights are fixed, hence for linear operators $l$ can be easily computed.
	\begin{algorithm}[H]
		\begin{algorithmic}
			\Require {Choose $p_0 \in M_0$, $\delta>0$, threshold parameter $\tau$, maximum number of iterations $K$.}
			\Ensure  {A sequence of points $\{p_s\}_{s=1,\dots,K}$ approximatively in $[p_0]$; The energy $E$ of the approximating polygonal}
			\State{Initialise the energy: $E\gets 0$}
			\For {$k=1,\dots,K-1$}
			\State {Compute $g_{\mathcal{N}(p_k)}^n$}
			\State {Compute the pullback metric $g^0_{p_k}$ trough \Cref{eq:pullback}}
			
			\If {For some $i,j=1,\cdots,dim(M_0) \ |{g^0_{p_k}}_{ij} - {g^0_{p_{k-1}}}_{ij}|>\tau$}
			\State {End loop}
			\EndIf			
			
			\State {Diagonalize $g^0_{p_k}$ and find the eigenvectors}
			\State {$v_k \gets$ Choose a random linear combinations of the null eigenvectors}
			\State {Compute the new point $p_{k+1} \leftarrow p_{k}+ \delta v_{k}$}
			\State {Add the contribute of the new segment to the energy $E$ of the polygonal}
			\EndFor
		\end{algorithmic}
		\caption{Modified SiMEC $nD$ random walk algorithm with non differentiable activation functions}\label{al:SIMECnD_modified}
	\end{algorithm}
	
	This algorithm can be employed also to build random walks in the representation manifolds or in the input manifold of a neural network, in accordance with \Cref{prop:relu_network_equivalence_classes}.
	It is worth noticing, however, that this modification of \Cref{al:SIMECnD} is slower than the original version, since the continuity check is $\mathcal{O}(n_0^2)$, with $n_0 = dim(M_0)$. For the same reason the proposed algorithm is more subject to the curse of the dimensionality compared to the original one. Therefore in high dimensional spaces it may be more convenient to employ \Cref{al:SIMECnD} directly. This claim is also supported by the fact that the numerical errors in the construction of the polygonal approximating the null curve may not allow to distinguish between regions with different dimensions anyway -- In general due to the approximations, the curves we build can build lie in the union of several equivalence classes which are close to each other, namely the outputs of the network along the points of the curve are very close but not exactly the same.
	\begin{remark}
		Layers built using the leaky ReLU activation function can be treated in a similar fashion.
	\end{remark}
	
	\subsection{Activation functions with a jump discontinuity}\label{sec:jump}
	At last we briefly discuss how to extend the previous results to a certain class of activation functions with a jump discontinuity. We begin with an example showing that, in general, the pullback metric $f^*g$ carries very little information if we consider an activation function with a jump discontinuity.
	\begin{example}
		Consider a function $f:\mathbb{R}^2 \rightarrow R$ defined as $f := \Theta \circ A$ with $\Theta$ the Heaviside step function
		\begin{equation*}
		\Theta(x) =
		\begin{cases}
		1 \textit{ if } x \geq 0\\
		0 \textit{ if } x < 0
		\end{cases} \quad \textit{and} \quad
		A= 
		\begin{pmatrix}
		3 \\
		0 \\
		0
		\end{pmatrix}^T.
		\end{equation*}
		Since the derivative of $\Theta$ is zero except for $x=0$, where it does not exists, the pullback metric $f^*g$ is the null matrix almost everywhere, therefore the pullback metric is not detecting the jump of the activation function.
	\end{example}
	However, for the activation functions considered in \Cref{prop:null_continuous_1} the continuity hypothesis at the point in which we glue together the different maps can be relaxed. Indeed, we can admit a jump discontinuity and in general we can substitute the constant zero function with a constant non-zero function. Furthermore, we can consider also functions of the form
	$$\sigma(x) =
	\begin{cases}
	\alpha \textit{ if } -\infty \leq x < a\\
	h_1(x) \textit{ if } a \leq x \leq b\\
	\beta \textit{ if } x < a
	\end{cases}
	$$
	where $\alpha,\beta,a,b\in\rd$, $\alpha < \beta\, (\alpha > \beta)$, $a<b \, (a>b)$, with $\alpha \leq h_1(a) < h_1(b) \leq \beta$ $(\alpha~\geq~h_1(a)~>~h_1(b)~\geq~\beta)$,  and $h_1(x)$ a differentiable, strictly increasing (decreasing) function.
	In this particular case the statement of \Cref{prop:1d-monotone-funct} continues to hold true repeating the proof in the three regions. Furthermore, noting that the two regions in which $f$ is constant cannot be contiguous entails the analogous of \Cref{prop:null_continuous_1}. Furthermore, for this kind of activation functions, if we require that the right and left limits of $h_1$  and $h_2$ are not zero, then the results of \Cref{subsec:relu_layers} keep to hold true and the SiMEC and SiMExp algorithm can be applied.
	
	\section{Numerical experiments and applications}\label{sec:numerical_experiments}
	
	The numerical experiments in this section have been written using PyTorch 2.1.2, on a machine equipped with Linux Ubuntu 22.04.1 LTS, an Intel i7-10700k CPU, 24 GB of RAM and a NVIDIA RTX 2060 SUPER GPU. The code can be found at \url{https://github.com/alessiomarta/extension_singular_riemannian_framework_code}.

\subsection{Equivalence classes of ReLU and leaky ReLU layers}\label{subsec:relu_and_leaky_relu}
\paragraph{ReLU activation function}
We present some numerical experiments concerning the ReLU activation function, which is not differentiable at the origin. We begin with a numerical experiment in $\mathbb{R}^2$. Let us consider $f=\relu \circ A$ with 
$$
A= 
\begin{pmatrix}
1 \\
-1
\end{pmatrix}^T.
$$ 
We run the SiMEC algorithm for $5000$ steps, with $\delta = 1\cdot 10^{-2}$ and $\epsilon = 1 \cdot 10^{-8}$, starting from different points. The first point is $(-0.98,-2.45)$ , in which $\relu$ is in the active state. The dimension of the kernel of the pullback metric is $1$ and the equivalence class is a line (\cref{fig:relu_2d_activation}, yellow curve). All the points of this line are mapped by $f$ to the same output: the same happens for the point $(-1.5,-2)^\top$, whose equivalence class is the violet curve in \cref{fig:relu_2d_activation}. The other starting point is $(-1.45,1.30)^\top$, which does not activate the ReLU function. Starting from here, one has that $\dim \Ker(f^*g)=2$ and hence the equivalence class is the whole half-plane $x<y$, the gray area in \cref{fig:relu_2d_activation}. The blue curve depicts the equivalence class of $(1.45,1.30)^\top$: since the number of allowed step is 5000, only a small portion of its equivalence class is recovered (which is actually infinite). Choosing another point, namely $(0.2,1)^\top$, one obtains the orange curve in \cref{fig:relu_2d_activation}. Due to the intrinsic stochastic nature of the adopted random walk approach, different runs of SiMEC leads to different reconstructions of the (partial) equivalence class in the inactive region: \cref{fig:relu_2d_inactive} shows that 3 different runnings, starting from the same point and with the same initialisation, lead to different paths. As discussed at the beginning of \Cref{sec:numerical_reconstruction_non_differentiable}, \Cref{al:SIMECnD} and its modified version for ReLU layers yields similar results.
\begin{figure}
	\begin{center}
		\begin{subfigure}[t]{0.5\textwidth}
			\centering
			\includegraphics[width=.95\linewidth]{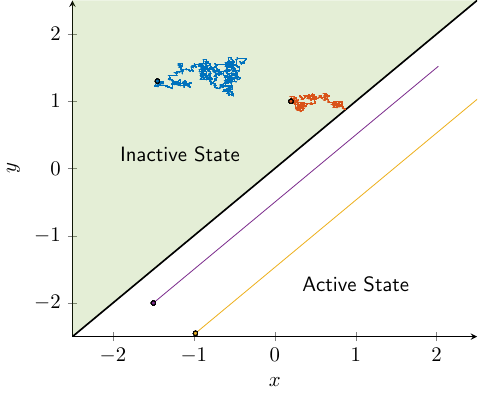}
			\caption{Equivalence classes of 4 different points.}
			\label{fig:relu_2d_activation}
		\end{subfigure}%
		\begin{subfigure}[t]{0.5\textwidth}
			\centering
			\includegraphics[width=1.\linewidth]{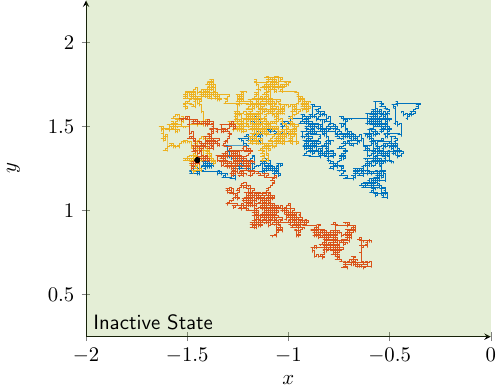}
			\caption{Different paths are recovered starting from the same point.}
			\label{fig:relu_2d_inactive}
		\end{subfigure}
	\end{center}
	\caption{SiMEC algorithm for ReLU function. The circular dots in the plots refer to the starting points of the algorithm. Left panel: the black solid line depicts the activation line for function $f$. For $x>y$ ReLU is in its active state, while for $x \leq y$ is inactive (gray area). Left panel: blue and orange paths depict the recovered equivalence class of $(1.45,1.30)^\top$ and $(0.2,1)^\top$, respectively. The violet and the yellow line consists in the equivalence classes of $(-0.98,-2.45)$ and $(-1.5,-2)^\top$, respectively. Right panel: 3 different SiMEC runnings with the same starting point. We refer to the online version of the work for color references.}
	\label{fig:relu_2d}
\end{figure}
Next we consider a 3-dimensional case. Let us consider $f=\relu \circ A$ with
$$
A= 
\begin{pmatrix}
1 \\
-1 \\
1
\end{pmatrix}^T.
$$ 
This time we run the SiMEC algorithm for $5000$ steps, with $\delta = 10^{-2}$ and $\epsilon = 1 \cdot 10^{-1}$, in order to see more clearly what happens near the starting points. The ReLU function is in the active state for the points $(x,y,z) \in \mathbb{R}^3$ such that $z>y-x$. We consider initially the point $(1.85,1.3,-1)^\top$, for which the ReLU function is inactive: in this 3 dimensional case, the equivalence class consists in the entire halfspace \emph{above} the plane $y-x-z=0$. In \cref{fig:relu3d_inactive1} we depict the path recovered via out procedure, together with the approximating plane obtained via a Least Square regression, shown for clarifying the visual inspection. Indeed, it is evident that all the points explore the space among all the 3 directions. 

Starting, instead, from the point $(1.85,1.30,1.50)$ we recognize that we are in the active region: indeed, by computing the regression plane of the points generated by the algorithm one obtains the plane $-x+y-z+2.05=0$, which is parallel to $-x+y-z=0$ (see \cref{fig:relu_3d_active}).

\begin{figure}
	\begin{subfigure}[t]{0.32\textwidth}
		\centering
		\includegraphics[height=0.25\textheight]{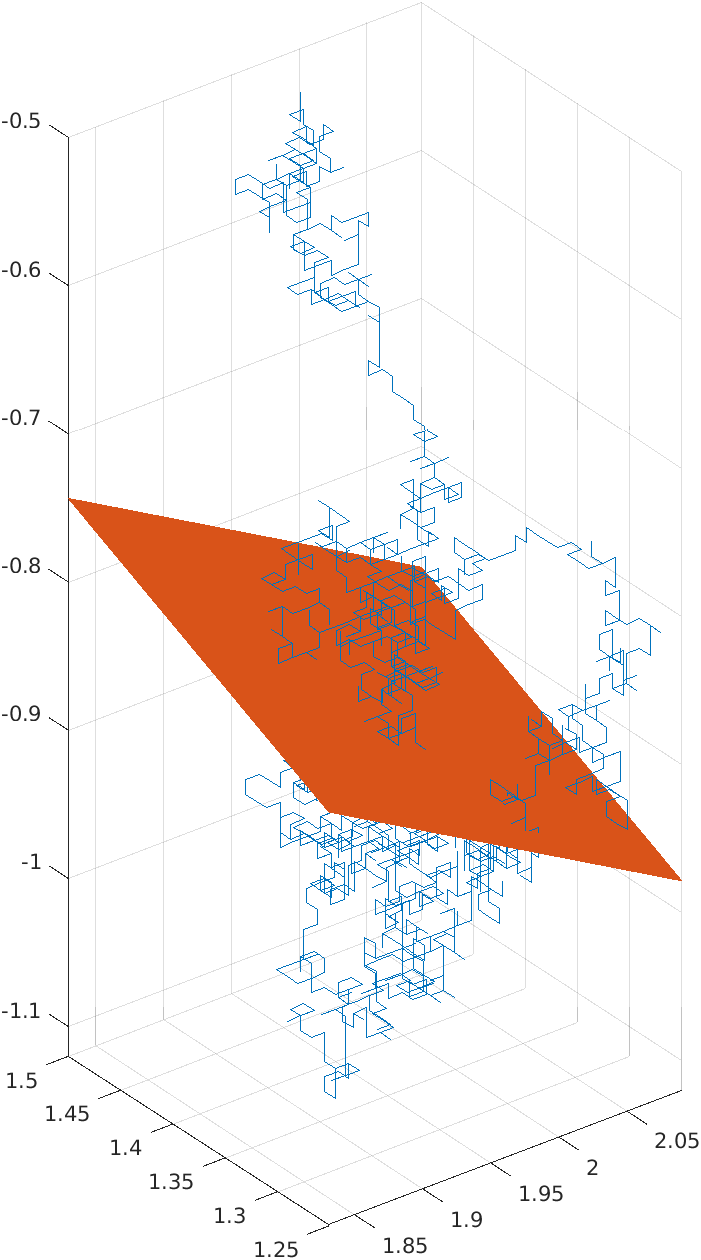}
		\caption{Equivalence class when ReLU is inactive.}
		\label{fig:relu3d_inactive1}
	\end{subfigure}\hfill\begin{subfigure}[t]{0.32\textwidth}
		\centering
		\includegraphics[height=0.25\textheight]{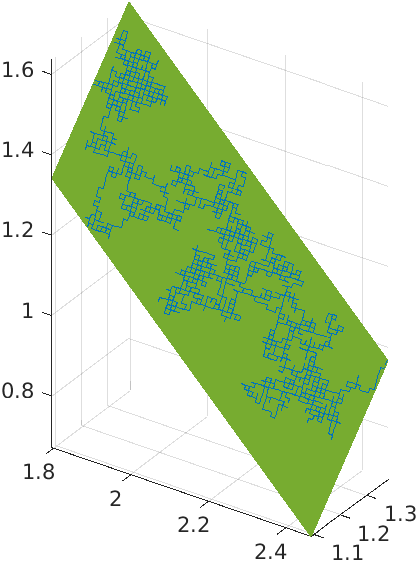}
		\caption{Equivalence class when ReLU is active.}
		\label{fig:relu_3d_active}
	\end{subfigure}\hfill\begin{subfigure}[t]{0.32\textwidth}
		\centering
		\includegraphics[height=0.25\textheight]{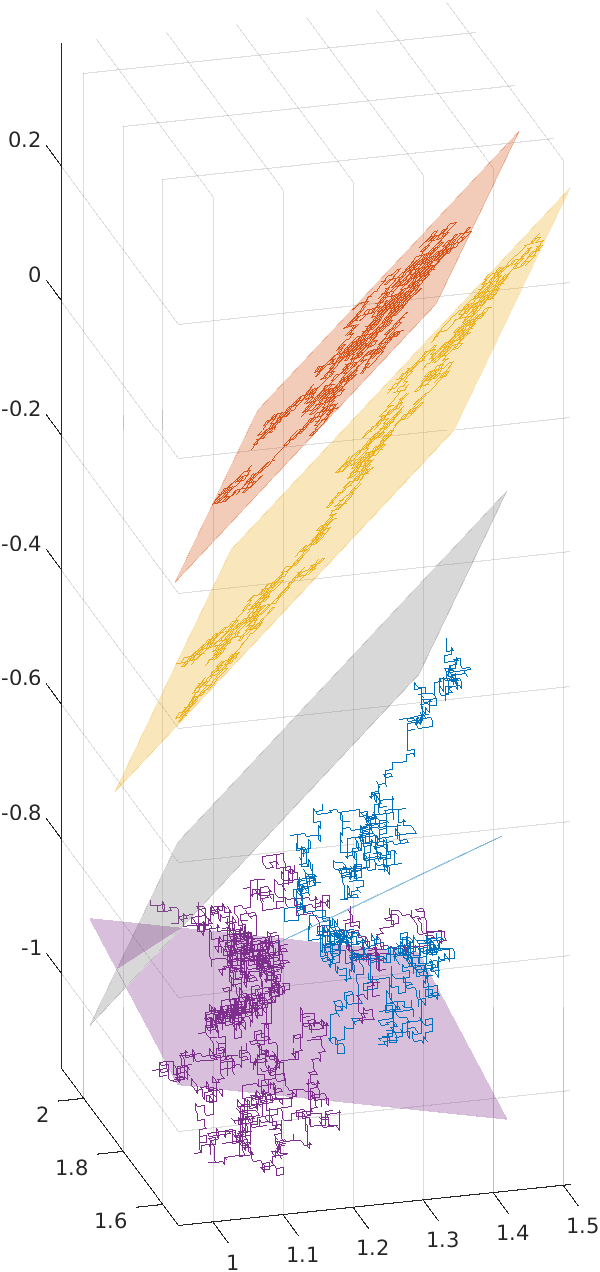}
		\caption{Gray plane is the separating plane.}
		\label{fig:relu_3d_all}
	\end{subfigure}
	\caption{ReLU's equivalence classes in 3D case. Left panel: when the starting point lies in the inactive region, then the equivalence classes consists in a the whole half space below the activation plane. we depict in orange the regression plane of the points generated by the algorithm for visualization purposes. Central panel: if the starting point lies in the active region, then the sequence generated by the algorithm lies on a plane, which is parallel to the activation plane. Right Panel: the gray plane depicts the separating plane $-x+y-z=0$. Blue and orange paths and plane refer to two different runs of the algorithm when the starting point is $()1.85,1.3,-1)^\top$, while the yellow and the violet ones are equivalence classes when the starting points are on the active region.}
	\label{fig:relu_3d_1}
\end{figure}

\paragraph{Leaky ReLU activation function}

The subsequent experimental part regards the leaky ReLU activation function, another map which is not differentiable at the origin. As for the previous case with ReLU, we begin with a numerical experiment in $\mathbb{R}^2$. Let us consider $f=\sigma \circ A$ with
\begin{equation*}
\sigma(x) =
\begin{cases}
-0.01x & \textit{ if } x < 0\\
x & \textit{ if } x \geq 0
\end{cases}, \quad 
A= 
\begin{pmatrix}
2 \\
-1
\end{pmatrix}^T.
\end{equation*}
We run the SiMEC algorithm for $10000$ steps, with $\delta = 1\cdot 10^{-3}$ and $\epsilon = 1 \cdot 10^{-8}$, starting first from the $(-1,2)^\top$. The equivalence class is depicted in blue in \cref{fig:leaky_relu_2d}: such class is a parallel line to $y=2x$, which is the linear operator described by the matrix $A$. Considering the point $(2,1.3)^\top$ one obtains again a line parallel to $y=2x$ (see \cref{fig:leaky_relu_2d}, orange line). The equivalence class is a line even in this case. Recall that such function has been employed as an approximation of $0$ for avoiding the vanishing gradient problem.
\begin{figure}[htbp]
	\begin{center}
		\begin{subfigure}[t]{0.45\textwidth}
			\includegraphics[width=1\textwidth]{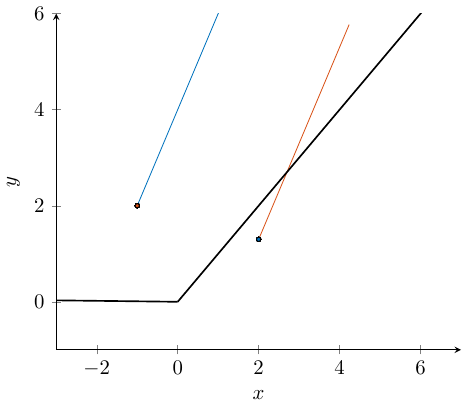}
			\caption{Leaky ReLU, $2D$ case.}
			\label{fig:leaky_relu_2d}
		\end{subfigure}\hfill\begin{subfigure}[t]{0.45\textwidth}
			\includegraphics[width=1\textwidth]{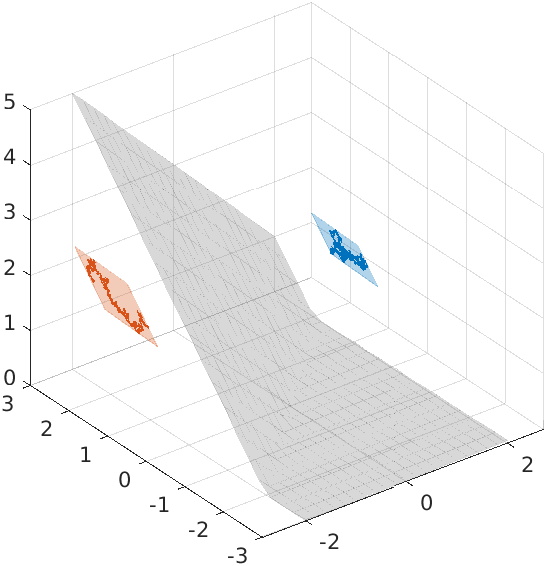}
			\caption{Leaky ReLU, $3D$ case.}
			\label{fig:leaky_relu_3d_1}
		\end{subfigure}
	\end{center}
	\caption{Results on leaky ReLU layers. Left panel: $2D$ case. The blue curve is a part of the equivalence class of $(1.85,1.30)$, while the orange one of the equivalence class of $(-1.85,1.30)$. In this case the equivalence classes are all parallel lines. Right panel: $3D$ case. It is evident that all the points trigger the activation function, and the algorithm provides paths that lie on parallel planes.}
	\label{fig:leaky_relu}
\end{figure}

As already done for the ReLU function, we move to the 3D case, considering the $f=\sigma \circ A$ with 
\begin{equation*}
\sigma(x) =
\begin{cases}
-0.01x & \textit{ if } x < 0\\
x & \textit{ if } x \geq 0
\end{cases}, \quad 
A= 
\begin{pmatrix}
1 \\
-1\\
1
\end{pmatrix}^T,
\end{equation*}
which is described by the same linear operator but the activation function is now the Leaky ReLU. As already observed in the two dimensional case, the equivalence classes are parallel planes: \cref{fig:leaky_relu_3d_1} depicts the recovered paths of the points $(-1.85,1.3,1.5)^\top$ (orange) and $(1.85,1.3,1.5)^\top$ (blue), together with the separating surface (gray).

\subsection{Equivalence classes in the MNIST handwritten digits dataset}\label{subsec:exp_mnist}

In this section we use the framework proposed in \cref{sec:extension_non_differentiable,sec:extension_c1_layers} to study the equivalence classes in the MNIST handwritten digits dataset, which in this case correspond to the set of all the images classified with the same probability as the same digit. In this numerical experiment we employ the network depicted in \cref{fig:mnist_network}: two subsequent blocks consisting in convolutional, average pool layer and ReLU layers are followed by a vectorization operation. Then, two couples of fully connected plus ReLU layers are spaced by a dropout layer. Eventually, a fully connected layer followed by a softmax one provides the final classification probabilities. The dimension of the filters of the convolutional layers is $5\times 5$, whilst the average pool layers are halving the spatial dimensions of the input.
\begin{figure}[htbp]
	\begin{center}
		\includegraphics[width=\textwidth]{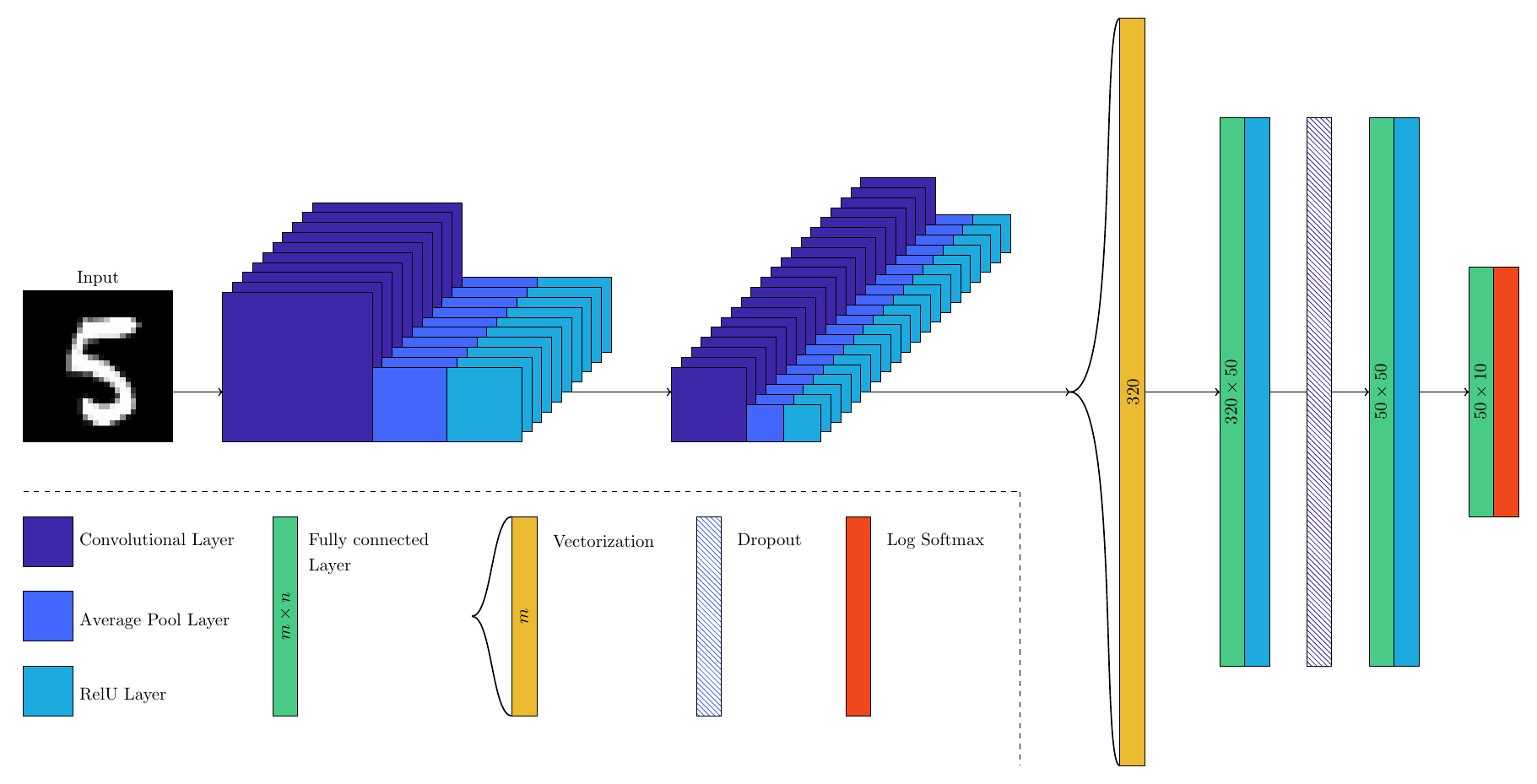}
	\end{center}
	\caption{The neural network used in the numerical experiments with the MNIST handwritten digits dataset. The vectorization operation transforms the output of the second ReLU layer into a vector with 320 components; the fully connected layers are characterized by the dimension $m\times n$.}
	\label{fig:mnist_network}
\end{figure}
As a first step, the training process consists in 10 epochs,  arriving to a $98 \%$ accuracy, minimizing the negative log likelihood loss via the SGD method with a learning rate of 0.01 and a momentum of 0.5. The first experiment we perform with the MNIST handwritten digits dataset \cite{mnistdataset} is the reconstruction of equivalence classes -- namely pictures of digits classified in the same way. To this end, we run \cref{al:SIMECnD} for $1000$ steps with $\epsilon=10^{-5}$ and $\delta=10^{-4}$, starting from images of the digit zero and of the digit five. The results can be seen in \Cref{fig:mnist_null_direction}.
\begin{figure}[h]
	\begin{center}
		\begin{subfigure}[t]{0.45\textwidth}
			\includegraphics[width=\textwidth]{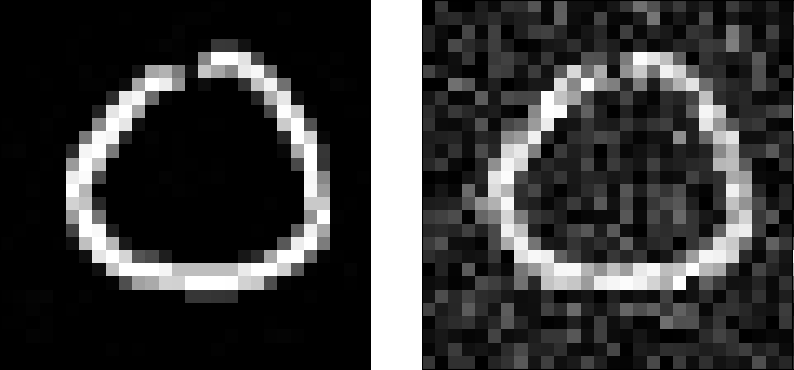}
			\caption{Images of 0 recognized to be in the same equivalence class: they are both classified as zeroes with a $99.78114 \%$ probability.}
		\end{subfigure}\hfill\begin{subfigure}[t]{0.45\textwidth}
			\includegraphics[width=\textwidth]{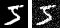}
			\caption{Images of 5 recognized to be in the same equivalence class: they are both classified as fives with a $99.63834 \%$ probability.}
		\end{subfigure}
	\end{center}
	\caption{Proceeding along null directions does not change the output of the network. A random walk built using the eigenvectors associated to the null vectors can be employed to see how much a network is resistant to the presence of noise in the pictures.}
	\label{fig:mnist_null_direction}
\end{figure}
As expected the images produced by the algorithm are classified in the same way -- if the parameter $\delta$ employed to build the polygonal approximating the null curve is small enough. 

Next, we try to modify the picture of a digit proceeding along the non-null directions, which, according to the results in \cref{sec:geometric_framework}, correspond to move among different equivalence classes. Employing \Cref{al:SIMEXPnD} for $1000$ steps with $\epsilon=10^{-5}$ and $\delta=10^{-4}$, we start from an image of the digit four and we arrive to a picture that may represent the digit nine. Indeed, this is confirmed by the output of the network which classifies the final step of the path built by SiMEXP as 9. A visual inspection of some iterates fo the algorithm are depicted in \cref{fig:mnist_null_direction}.
\begin{figure}[h]
	\begin{center}
		\includegraphics[width=\textwidth]{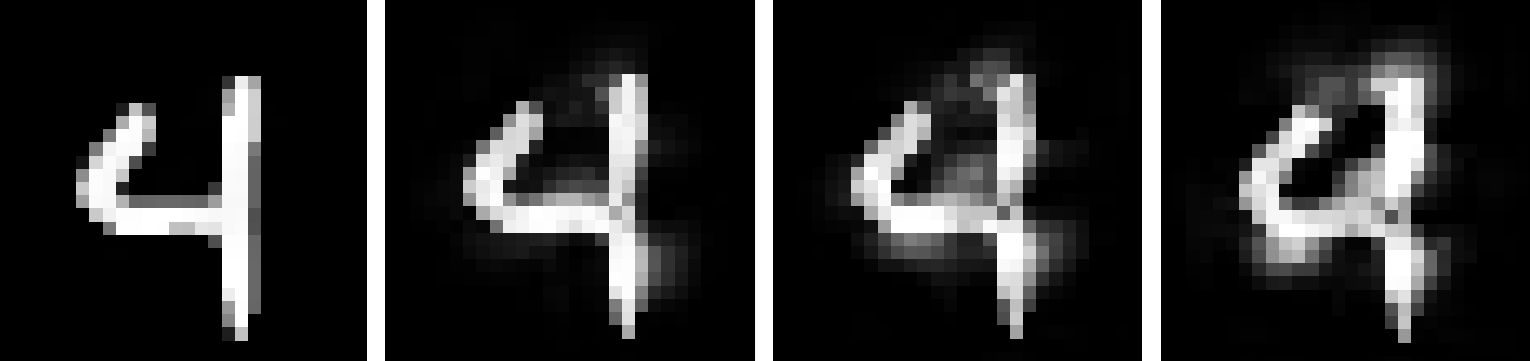}
	\end{center}
	\caption{Proceeding along non-null directions does change the output of the network. After a few iterations a random walk built using the eigenvectors associated to the null vectors transforms the image of the digit four in a picture which is resembling a nine. This can be seen also looking at the output of the network. The first picture is classified as a $4$ with a probability of $99.69651 \%$ and as a nine with probability $5.32 \cdot 10^{-4}$. While we run the algorithm the probability to be classified a a 4 decreases. After $1000$ iterations we arrive to the last picture, which is classified as a $4$ with a probability of $29.71972 \%$ and as a $9$ with probability $70.54823 \%$.}
	\label{fig:mnist_positive_direction}
\end{figure}

\subsection{Level sets in a $4D$ time series}
\label{subsec:exp_ccpp}

Now we build equivalence classes in a regression task over a time-series, training a network over the Combined Cycle Power Plant dataset \cite{combined_cycle_power_plant}. This dataset contains $9568$ points collected between 2006 and 2011 from a power plant working at full load. The features consist of hourly average ambient variables: Temperature (in $^\circ C$), exhaust vacuum (in $cm \ Hg$), ambient pressure (in $millibar$) and relative humidity (dimensionless). To each of these quadruplets, the associated electrical energy output of the plant is recorded. The regression task is to predict the energy output (in $MW$) corresponding to the features. In \cite{electronics12112431} this goal was accomplished making use of a deep transformer encoder containing fully connected layers, residual blocks and a self-attention block. Since our goal is to test the theory developed in this paper, we use a simplified version of this network and we substitute the self-attention block with a LSTM unit -- which will roughly play the same role, namely finding patterns between different points in a sequence. The diagram of our network is depicted in \Cref{fig:ccpp_network}.
\begin{figure}
	
	\begin{center}
		\includegraphics[width=\textwidth]{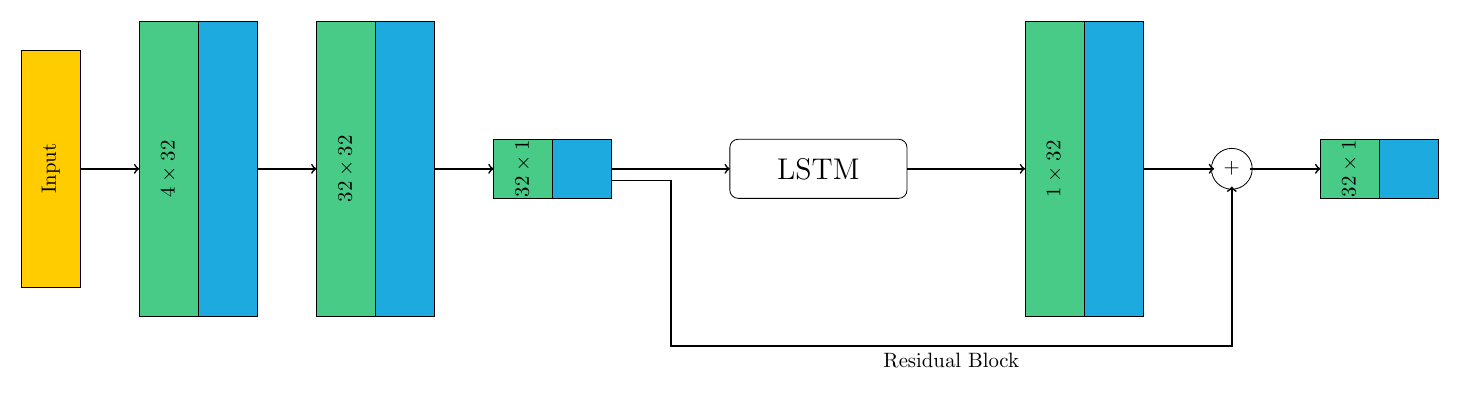}
	\end{center}
	\caption{The simple deep neural network we employ in this numerical experiment, with three fully connected layers, a residual block and a LSTM unit. $x$ and $y$ represent the input (of dimension $4$) and the output (of dimension $1$). $F_1,F_2,F_3,F_4,F_5$, are ReLU fully connected layers of dimensions -- input $\times$ output -- $4 \times 32$, $32 \times 32$, $32 \times 1$, $1 \times 32$ and $32 \times 1$ respectively.}
	\label{fig:ccpp_network}
\end{figure}
We run \Cref{al:SIMECnD} for $1000$ steps with $\epsilon=10^{-7}$ and $\delta=10^{-4}$, starting from a point whose ambient variables are $23.64 \ {}^\circ C, \ 58.49 \ millibar, \ 1011,4 \ mm \ Hg$ and a relative humidity of $74.2 \%$. For these values the network predicts an energy output of $469.69 \ MW$. A random walk in the equivalence class of this point built by the algorithm is illustrated in \Cref{fig:random_walk_ccpp}, via 2d and 3d projections.

\begin{figure}[htbp]
	\begin{subfigure}[t]{0.4\textwidth}
		\begin{center}
			\includegraphics[width=1\linewidth]{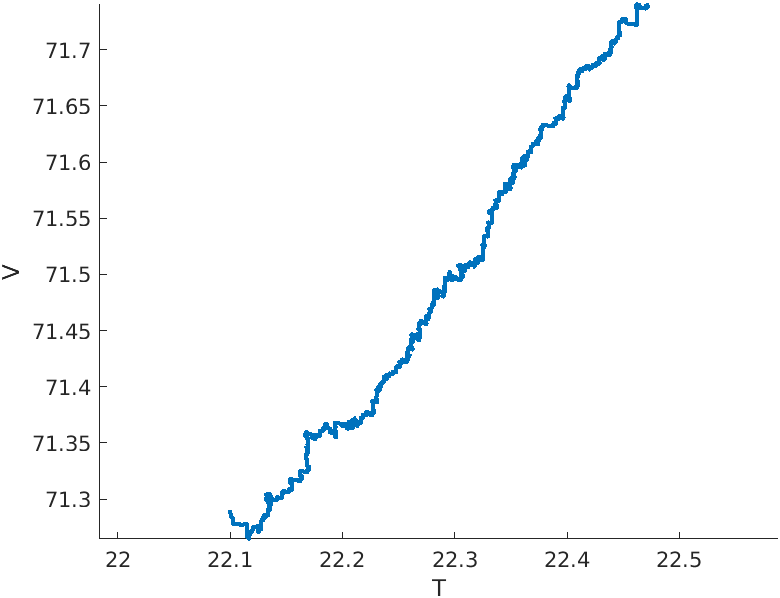}
		\end{center}
		\caption{2D projection ($T-V$) of the recovered equivalence class.}
	\end{subfigure} \hfill\begin{subfigure}[t]{0.4\textwidth}
		\begin{center}
			\includegraphics[width=1\linewidth]{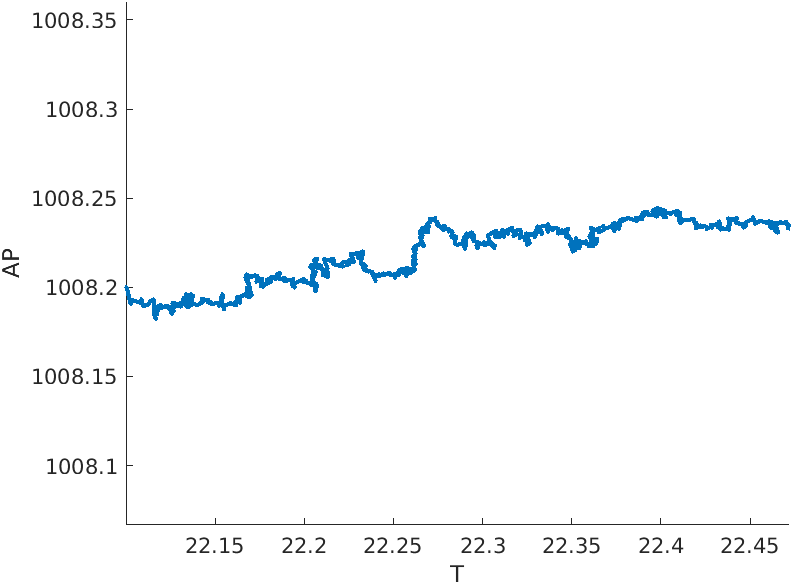}
		\end{center}
		\caption{2D projection ($T-AP$) of the recovered equivalence class.}
	\end{subfigure}
	\begin{subfigure}[t]{1\textwidth}
		\begin{center}
			\includegraphics[width=.9\linewidth]{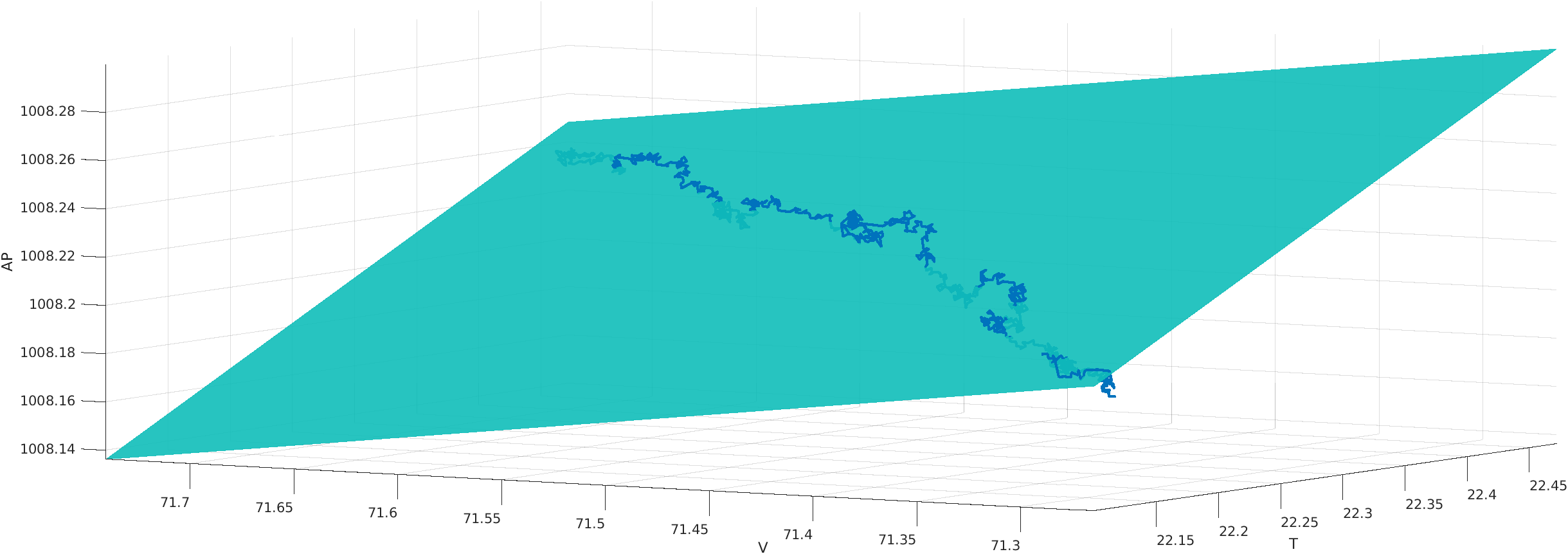}
		\end{center}
	\end{subfigure}
	\caption{Top panels: two $2$-D projections of a random walk in the equivalence class of the point $P_0=(23.64,58.49,1011,4,74.2)$ in the features space. Bottom panel $3$-D projection on the $T,V,AP$ space of the random walk we built in the equivalence class of the same point $P_0=(23.64,58.49,1011,4,74.2)$. The prediction associated by the network to all the points of this random walk is always $469.69 \ MW$.}
	\label{fig:random_walk_ccpp}
\end{figure}

\section{Conclusions}
\label{sec:concl}
In this work we extended the results of \cite{BeMa21a} to more complex network architectures and to a more general class of functions, furthermore we developed the new tailored version for multidimensional problems of SiMEC and SiMExp algorithms of \cite{BeMa21b}. Practical applications, namely to image classification problems and to time series problem in power plants problem, show that this approach provides insights on the behaviour of deep networks. Further research can be done, bridging these results with XAI methods, in order to provide explanation of AI behaviour in particular instances, such as adversarial attacks.

\printbibliography

\end{document}